\documentclass[12pt]{amsart}
\usepackage{stmaryrd}
 \usepackage{amsmath}
 \usepackage{amsfonts}
 \usepackage{amssymb}
 \usepackage{amsthm}
 \usepackage{algorithm}
 \usepackage{algpseudocode}

 \usepackage[all,cmtip]{xy}           
\usepackage{xspace}
\usepackage{bbding}
\usepackage{txfonts}
\usepackage[shortlabels]{enumitem}
\usepackage{ifpdf}
\ifpdf
  \usepackage[colorlinks,final,backref=page,hyperindex]{hyperref}
\else
  \usepackage[colorlinks,final,backref=page,hyperindex,hypertex]{hyperref}
\fi
\usepackage{tikz}
\usepackage{tkz-euclide}
\usepackage[active]{srcltx}
\usepackage{multicol}

\makeatletter
\@namedef{subjclassname@2020}{\textup{2020} Mathematics Subject Classification}
\makeatother

 \usepackage{verbatim}
 \usepackage{xcolor}
 \usepackage[english]{babel}
 \usepackage[margin= 1.5in]{geometry}
 \newtheorem{theorem}{Theorem}
 \theoremstyle{definition}
 \newtheorem{example}{Example}
 
 \newcommand{\R}{\mathbb{R}}

 \title{A Reduction Algorithm for Volterra Integral Equations}
\author{Richard Gustavson}
\address{Department of Mathematics, Manhattan College, Riverdale, NY 10471, United States}
\email{rgustavson01@manhattan.edu}

\author{Sarah Rosen}
\address{Department of Mathematics, Manhattan College, Riverdale, NY 10471, United States}
\email{srosen01@manhattan.edu}
 \date{\today}

\begin{document}

\subjclass[2020]{
45D05,  
45P05,   
05C05,   
05C85,   
17B38   
}

\keywords{Integral equation, iterated integral, Volterra operator, rooted trees, reduction algorithm}

\maketitle

\begin{center}
    \textbf{Abstract}
\end{center}

An integral equation is a way to encapsulate the relationships between a function and its integrals.  We develop a systematic way of describing Volterra integral equations -- specifically an algorithm that reduces any separable Volterra integral equation into an equivalent one in operator-linear form, i.e. one that only contains iterated integrals.  This serves to standardize the presentation of such integral equations so as to only consider those containing iterated integrals.  We use the algebraic object of the integral operator, the twisted Rota-Baxter identity, and vertex-edge decorated rooted trees to construct our algorithm.

\section{Introduction}
Integral equations lie at an intersection of interest among pure and applied mathematicians. Practically, they represent the relationship between a function and its integrals, and they prove to be a useful tool in modeling real-world situations~\cite{Wazwaz, Zemyan}.  Applications of integral equations are incredibly diverse, as they are the natural analog of more familiar differential equations. Integral equations appear naturally in physics, in the form of the Electric Field Integral equation (and Magnetic Field Integral equation), as well as diffraction/scattering problems of both light and quantum particles~\cite{Brown, Olshevsky}. The Ornstein–Zernike integral equation, too, has applications in chemistry by describing the motion of fluids on a molecular level~\cite{Aebersold}.  In particular, iterated integrals have applications to, for example, topology, number theory, and quantum field theory~\cite{Chen1,Chen2}.  An important question to ask is when a given integral equation is equivalent to one containing only iterated integrals.

Integral equations in the context of analysis are used as a tool to work with infinities. Integrals are, after all, the sum of infinitely many infinitely small increments of the area under a certain function.  By considering an integral operator, that is, a function $I:C(\R) \to C(\R)$ defined by
\[
I(f)(x) := \int_a^xf(t)\,dt,
\]
we are able to work with integrals in a discrete sense. Using only algebraic properties and operations, we are able to simplify complicated combinations of integrals.  Such algebraic study of integral equations has been more recent.  One approach is using the Rota-Baxter operator, which is a generalization of the standard integral operator~\cite{Guo}.  These operators fail to represent integral equations with general kernels.  In~\cite{GGL}, an algebraic framework for arbitrary integral equations was produced using decorated rooted trees.  In addition, it is shown using non-constructive means that any Volterra integral equation with separable kernels is equivalent to one that is operator linear, that is, contains only iterated integrals.  In this paper, we construct an algorithm that performs this reduction.  This algorithm works by considering a separable Volterra integral equation as a decorated rooted tree, then performing the reduction on the tree.  The output tree represents the equivalent operator-linear form of the original integral equation.

\section{Background on Integral Equations}

An integral equation is an equation in which an unknown function, $u(x)$, depends on the integral of itself.
The most standard integral equation is of the form
\[
    u(x) = f(x) + \lambda \int_{a(x)}^{b(x)} K(x,t)u(t)\,dt
\]
where $a(x)$ and $b(x)$ are the limits of integration, $\lambda$ is a constant parameter, and $K(x,t)$ is a known function called the {\bf kernel}.  If the limits of integration are fixed, the integral equation is called a {\bf Fredholm} integral equation, whereas if at least one limit is a variable, the equation is called a {\bf Volterra} integral equation.  Throughout this paper, we will focus our discussion on Volterra integral equations with fixed lower limit 0 and upper limit $x$.  Examples of such integral equations are
\begin{align}
f(x) &= \left(\int_0^x e^{-x+t} f(t)\,dt\right)\left(\int_0^x \cos(t)g(t)\,dt\right) \label{notoperatorlinear}\\
f(x) &= \int_0^x \sin(x-t)(f(t))^2\left(\int_0^t e^{tu}g(u)\,du\right)\,dt. \label{operatorlinear}
\end{align}

An integral equation is said to be {\bf operator linear} if it does not contain any products of integrals.  For example, Eq.~\eqref{notoperatorlinear} is not operator linear, since the two integrals are multiplied together, whereas Eq.~\eqref{operatorlinear} is operator linear, as the integrals are iterated.

In order to study integral equations algebraically, we introduce the {\bf (Volterra) integral operator} $P_K:C(\R) \to C(\R)$ defined by
\begin{equation}\label{eq:intop}
P_K(f)(x) := \int_0^x K(x,t)f(t)\,dt.
\end{equation}
There is a Volterra integral operator corresponding to each kernel $K(x,t)$.  For example, Eq.~\eqref{notoperatorlinear} can be written as
\begin{equation}\label{noloperator}
f=P_{K_1}(f)P_{K_2}(g),
\end{equation}
where
$K_1(x,t) = e^{-x+t}$ and $K_2(x,t) = \cos(t)$, while Eq.~\eqref{operatorlinear} can be written as
\begin{equation}\label{oloperator}
f=P_{K_3}(f^2P_{K_4}(g)),
\end{equation}
where $K_3(x,t) = \sin(x-t)$ and $K_4(x,t) = e^{xt}$.  Note how the variables used in $K_4$ change in Eq.~\eqref{operatorlinear}; this is due to the iterated nature of the integrals.  One of the advantages of writing integral equations in operator form, as in Eqs.~\eqref{noloperator} and~\eqref{oloperator}, is that the specific variable names get absorbed in the operators.  In general, we study integral equations algebraically by moving all integral operators to one side of the equation and setting this equal to zero.  For example, Eq.~\eqref{noloperator} becomes $P_{K_1}(f)P_{K_2}(g)-f=0$.  In this case the left-hand side of such an equation is called an {\bf integral polynomial}.  We say two integral equations are {\bf equivalent} if they have the same solution set.  For example, the two equations
\begin{align}
f &= \left(\int_0^xf(t)\,dt\right)\left(\int_0^xg(t)\,dt\right) \label{product}\\
f &= \int_0^x f(t)\left(\int_0^tg(u)\,du\right)\,dt + \int_0^xg(t)\left(\int_0^tf(u)\,du\right) \label{IBP}\,dt
\end{align}
are equivalent since the right-hand sides of the equations are equal using integration by parts.  In a similar manner, we can define when two integral polynomials are equivalent.  Note that Eq.~\eqref{product} is not operator linear, whereas Eq.~\eqref{IBP} is, showing in this particular case that an integral equation that is not operator linear is equivalent to an operator linear equation.

An integral operator is said to be {\bf separable} if the kernel is separable, that is, if it can be written as a product of two single-variable functions $K(x,t) = k(x)h(t)$.  For example, $K_1(x,t) = e^{-x+t} = e^{-x}e^t$ above is separable, whereas $K_4(x,t) = e^{xt}$ above is not separable.  An integral equation is separable if every integral operator in the equation is.  Our goal in this paper is to construct an algorithm that will transform any separable Volterra integral equation into an equivalent (that is, having the same solution set) operator linear form.  To do so, we must determine algebraic identities that are satisfied by Volterra integral operators.

A \textbf{Rota-Baxter operator} is a linear function $R:C(\R) \to C(\R)$ such that for all $f,g \in C(\R)$,
\[
R(f)R(g) = R(R(f)g) + R(f\,R(g)).
\]
If we set the kernel of the integral operator in Eq.~\eqref{eq:intop} to 1, we can define a simplified integral operator $I:C(\R) \to C(\R)$ by
\[
I(f)(x) := \int_0^x f(t)\,dt.
\]
It is straightforward to check, using the integration-by-parts identity, that $I$ is a Rota-Baxter operator \cite{Guo}.  We can also define {\bf matching Rota-Baxter operators} \cite{GGZ} to study Volterra operators with multiple kernels, all of the form $K(t)$, where $t$ is the variable of integration.

When the kernel is a function of both the limit variable $x$ and the variable of integration $t$, the integral operator $P_K$ is not a Rota-Baxter operator.  For example, let $K(x,t) = x$ and $f=g=1$, then 
    \[
    P_K(f)(x)P_K(g)(x)=x^4
    \]
    whereas 
    \[
    P_K(f\,P_K(g))(x)+P_K(P_K(f)g)(x) = \frac{2}{3}x^4.
    \]
We can generalize the concept of a Rota-Baxter operator as follows.  A \textbf{twisted Rota-Baxter operator} is a linear operator $P:C(\R) \to C(\R)$ together with an invertible element $\tau \in C(\R)$ (i.e. $1/\tau \in C(\R)$ exists), called the {\bf twist}, such that for all $f,g \in C(\R)$,
    \[
    P(f)P(g)=\tau P(\tau^{-1} f\,P(g)) + \tau P(\tau^{-1}P(f)g).
    \]

We will show that a separable Volterra integral operator is a twisted Rota-Baxter operator.  Since we want to work with integral equations having multiple different kernels, we introduce the following generalization.  Let $\Omega$ be an indexing set, and for each $\omega \in \Omega$, let $P_\omega:C(\R) \to C(\R)$ be a linear operator and $\tau_\omega \in C(\R)$ be an invertible element.  We say that the collection $P_\Omega := \{P_\omega \mid \omega \in \Omega\}$ is a {\bf matching twisted Rota-Baxter operator} with {\bf twist} $\tau_\Omega := \{\tau_\omega \mid \omega \in \Omega\}$ if for all $f,g \in C(\R)$ and all $\alpha, \beta \in \Omega$, we have
\begin{equation}\label{MTRBA}
P_\alpha(f)P_\beta(g) = \tau_\alpha P_\beta(\tau_\alpha^{-1}P_\alpha(f)g) + \tau_\beta P_\alpha(\tau_\beta^{-1}fP_\beta(g)).
\end{equation}

\begin{theorem}\label{theorem:MTRBA}
\cite{GGL}
For each $\omega \in \Omega$ an indexing set, let $K_\omega(x,t) = k_\omega(x)h_\omega(t) \in C(\R^2)$ be a separable kernel with $k_\omega(x) \neq 0$ for all $x$, and $\tau_\omega = \frac{k_\omega(x)}{k_\omega(0)} \in C(\R)$.  Let $P_\omega := P_{K_\omega}:C(\R) \to C(\R)$ be the separable Volterra integral operator defined in Eq.~\eqref{eq:intop}.  Then $P_\Omega :=\{P_\omega \mid \omega \in \Omega\}$ is a matching twisted Rota-Baxter operator with twist $\tau_\Omega$.
\end{theorem}

Given an indexing set $\Omega$ and a collection of separable kernels $K_\Omega = \{K_\omega(x,t) \mid \omega \in \Omega\}$, we will write $\int_0^x K_\omega(x,t)f(t)\,dt$ in either integral form as $\int_\omega f$ or operator form as $P_\omega(f)$, where $P_\omega :=P_{K_\omega}$ is defined as in Eq.~\eqref{eq:intop}.  Thus in integral form, Eq.~\eqref{MTRBA} becomes
\begin{equation}\label{eq:RBint}
\left(\int_\alpha f\right)\left(\int_\beta g\right) = \tau_\beta \int_\alpha \left(\tau_\beta^{-1} f\left(\int_\beta g\right)\right) + \tau_\alpha \int_\beta\left(\tau_\alpha^{-1}\left(\int_\alpha f\right)g\right).
\end{equation}

Theorem~\ref{theorem:MTRBA} is used to prove:

\begin{theorem}\label{theorem:equivalence}
\cite{GGL}
Every separable Volterra integral equation is equivalent to one that is operator linear.
\end{theorem}

The proof of Theorem~\ref{theorem:equivalence} in \cite{GGL} is non-constructive.  It is possible to transform a specific separable Volterra integral equation into an equivalent operator linear form using Eq.~\eqref{MTRBA} in an ad hoc manner.  In this paper we will present an algorithm that is guaranteed to perform this transformation, giving us a constructive proof of Theorem~\ref{theorem:equivalence}.

\section{Decorated Rooted Trees}
We now introduce the concept of a {\bf decorated rooted tree}.  For more information on trees and many of the concepts used in this section, see~\cite{Diestel}.  A {\bf rooted tree} is a finite graph with no loops that has a distinguished ``root" vertex.  The {\bf height} of a vertex in a rooted tree is the number of edges in the unique path connecting that vertex to the root vertex.  The height of a rooted tree is the maximum height among all of the vertices in the tree.  While we do not include arrows on these trees to orient direction, the convention we use is that the bottom vertex in each tree is the root of that tree.

The {\bf parent} of a vertex $v$ in a rooted tree $T$ is the vertex adjacent to $v$ on the path connecting $v$ to the root vertex.  It is clear that every non-root vertex has a unique parent, while the root vertex has no parent.  A {\bf child} of a vertex $v$ is any vertex that has $v$ as its parent.  A vertex can have any (finite) number of children.  A vertex with no children is called a {\bf leaf}.  A vertex with multiple children is called a {\bf branching point}.  Note that every rooted tree contains at least one leaf (except the tree consisting solely of the root vertex, which is never considered a leaf, and no edges).  A subtree containing no branching points is called a {\bf chain}.  In a similar way we can define an {\bf ancestor} of a vertex $v$ to be any vertex in the path from $v$ to the root, and a {\bf descendant} of $v$ to be any vertex $w$ that contains $v$ on the path from $w$ to the root.  The root vertex is an ancestor of every non-root vertex, and every non-root vertex is a descendant of the root vertex.  

For example, consider the tree $T$ below.
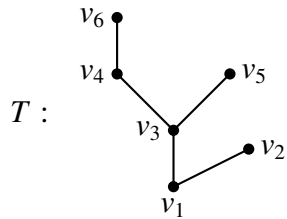
\[
\begin{tikzpicture}[thick]
\coordinate [label=below: $v_1$] (A) at (0,0);
\coordinate [label=left: $v_3$] (B) at (0,0.75);
\coordinate [label=right: $v_2$] (C) at (1,0.5);
\coordinate [label=left: $v_4$] (D) at (-0.75,1.5);
\coordinate [label=left: $v_6$] (E) at (-0.75,2.25);
\coordinate [label=right: $v_5$] (F) at (0.75,1.5);
\draw (C) -- (A) -- (B);
\draw (F) -- (B) -- (D) -- (E);
\foreach \n in {A,B,C,D,E,F} \node at (\n)[circle,fill,inner sep=1.5pt]{};
\coordinate [label=left: $T:$] (G) at (-1.5,1);
\end{tikzpicture}
\]

This tree has root vertex $v_1$ and leaves $v_2$, $v_5$, and $v_6$.  Vertices $v_1$ and $v_3$ are branching points.  Vertex $v_3$ has children $v_4$ and $v_5$, and has descendants $v_4$, $v_5$, and $v_6$.  Vertex $v_6$ has parent $v_4$, and has ancestors $v_1$, $v_3$, and $v_4$.

Given a vertex $v$ in a tree $T$, define $T(v)$ to be the subtree of $T$ consisting of $v$ and all of its descendants (and the edges connecting them), having $v$ as the root.  Define $B(v)$ to be the subtree of $T$ consisting of all vertices of $T$ (and the edges connecting them) {\it except} the descendants of $v$.  For example, the subtrees $T(v_3)$ and $B(v_3)$ of the tree $T$ above are:

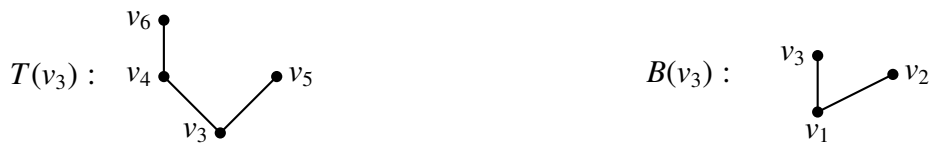
\begin{multicols}{2}
\[
\begin{tikzpicture}[thick]
\coordinate [label=left: $v_3$] (B) at (0,0.75);
\coordinate [label=left: $v_4$] (D) at (-0.75,1.5);
\coordinate [label=left: $v_6$] (E) at (-0.75,2.25);
\coordinate [label=right: $v_5$] (F) at (0.75,1.5);
\draw (F) -- (B) -- (D) -- (E);
\foreach \n in {B,D,E,F} \node at (\n)[circle,fill,inner sep=1.5pt]{};
\coordinate [label=left: $T(v_3):$] (G) at (-1.5,1.5);
\end{tikzpicture}
\]

\[
\begin{tikzpicture}[thick]
\coordinate [label=below: $v_1$] (A) at (0,0);
\coordinate [label=left: $v_3$] (B) at (0,0.75);
\coordinate [label=right: $v_2$] (C) at (1,0.5);
\draw (C) -- (A) -- (B);
\foreach \n in {A,B,C} \node at (\n)[circle,fill,inner sep=1.5pt]{};
\coordinate [label=left: $B(v_3):$] (G) at (-1,0.5);
\end{tikzpicture}
\]

\end{multicols}

A {\bf vertex-edge decorated rooted tree} is a rooted tree whose vertices and edges are labeled by elements of some set.  In our case, we label all vertices by (possibly unknown) elements of $C(\R)$ and label all edges by elements of $\Omega$, where $\Omega$ is some fixed indexing set (representing the set of kernels).  Let $\mathcal{E}(\Omega)$ be the set of all such vertex-edge decorated rooted trees, and let $\mathfrak{E}(\Omega)$ be the set of all formal sums of elements of $\mathcal{E}(\Omega)$.  An element of $\mathfrak{E}(\Omega)$ is called a {\bf forest}.  See~\cite{GGL} for a more formal description of $\mathfrak{E}(\Omega)$.

In~\cite{GGL}, it is shown that the set of forests $\mathfrak{E}(\Omega)$ is in one-to-one correspondence with the set of integral polynomials with set of kernels $\{K_\omega \mid \omega \in \Omega\}$.  Since every integral polynomial $p$ corresponds to an integral equation $p=0$, this is equivalent to saying that $\mathfrak{E}(\Omega)$ is in one-to-one correspondence with the set of all integral equations with kernels given by the $K_\omega$.  An edge labeled by $\omega \in \Omega$ connecting parent vertex labeled by $f \in C(\R)$ and child vertex labeled by $g \in C(\R)$ corresponds to $fP_\omega(g) = f \int_\omega g$.  A chain of length greater than one corresponds to an iterated integral.  For example, with $f,g,h \in C(\R)$ and $\alpha,\beta \in \Omega$, the tree
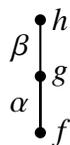
\[
\begin{tikzpicture}[thick]
\coordinate [label=right: $f$] (A) at (0,0);
\coordinate [label=right: $g$] (B) at (0,0.75);
\coordinate [label=right: $h$] (C) at (0,1.5);
\draw (C) -- (A) -- (B);
\foreach \n in {A,B,C} \node at (\n)[circle,fill,inner sep=1.5pt]{};
\coordinate [label=left: $\alpha$] (D) at (0,0.375);
\coordinate [label=left: $\beta$] (E) at (0,1.125);
\end{tikzpicture}
\]
corresponds to the integral
\[
fP_\alpha\left(gP_\beta(h)\right) = f\int_\alpha \left(g\int_\beta h\right).
\]
A branching point in a tree corresponds to multiplying integrals.  For example, with $f_1,\dots,f_5 \in C(\R)$ and $\alpha,\beta,\gamma,\delta \in \Omega$, the tree
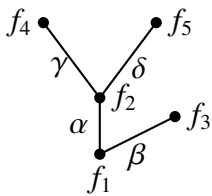
\[
\begin{tikzpicture}[thick]
\coordinate [label=below: $f_1$] (A) at (0,0);
\coordinate [label=right: $f_2$] (B) at (0,0.75);
\coordinate [label=right: $f_3$] (C) at (1,0.5);
\coordinate [label=left: $f_4$] (D) at (-0.75,1.75);
\coordinate [label=right: $f_5$] (E) at (0.75,1.75);
\draw (C) -- (A) -- (B);
\draw (E) -- (B) -- (D);
\foreach \n in {A,B,C,D,E} \node at (\n)[circle,fill,inner sep=1.5pt]{};
\coordinate [label=left: $\alpha$] (F) at (0,0.375);
\coordinate [label=below: $\beta$] (G) at (0.5,0.3);
\coordinate [label=left: $\gamma$] (H) at (-0.27,1.15);
\coordinate [label=right: $\delta$] (I) at (0.27,1.15);
\end{tikzpicture}
\]
corresponds to the integral
\[
f_1P_\alpha\left(f_2P_\gamma(f_4)P_\delta(f_5)\right)P_\beta(f_3) = f_1\left(\int_\alpha f_2\left(\int_\gamma f_4\right)\left(\int_\delta f_5\right)\right)\left(\int_\beta f_3\right).
\]

We define two operations on $\mathfrak{E}(\Omega)$ that correspond to operations on integrals.  The {\bf grafting product} of two decorated trees $T$ and $U$, denoted $T \veebar U$, is obtained by merging the roots of $T$ and $U$ into a common root shared by the branches of both trees.  The label of the root of $T \veebar U$ is the product of the labels of the roots of $T$ and $U$.  This is equivalent to multiplying two integrals together.  For example, we have:

\[
\begin{tikzpicture}[thick]
\coordinate [label=right: $a$] (A) at (0,0);
\coordinate [label=right: $b$] (B) at (0,1);
\coordinate [label=left: $\alpha$] (C) at (0,0.5);
\draw (A) -- (B);
\coordinate [label=left: $\veebar$] (D) at (1.25,0.5);
\coordinate [label=right: $c$] (E) at (2,0);
\coordinate [label=right: $d$] (F) at (2,1);
\coordinate [label=left: $\beta$] (G) at (2,0.5);
\draw (E) -- (F);
\coordinate [label=left: $\text{=}$] (H) at (3.25,0.5);
\coordinate [label=right: $ac$] (I) at (4.75,0);
\coordinate [label=left: $b$] (J) at (4,1);
\coordinate [label=right: $d$] (K) at (5.5,1);
\coordinate [label=left: $\alpha$] (L) at (4.45,0.5);
\coordinate [label=right: $\beta$] (M) at (5.1,0.5);
\draw (J) -- (I) -- (K);
\foreach \n in {A,B,E,F,I,J,K} \node at (\n)[circle,fill,inner sep=1.5pt]{};
\end{tikzpicture}
\]
This example is equivalent to the integral equation
\[
\left(a\int_\alpha b\right)\left(c\int_\beta d\right) = ac\left(\int_\alpha b\right)\left(\int_\beta d\right),
\]
or, in operator notation,
\[
\Big(aP_{\alpha}(b)\Big)\Big(cP_{\beta}(d)\Big) = ac\Big(P_{\alpha}(b)P_{\beta}(d)\Big).
\]

The other operation is the {\bf extension operator} $\Lambda_{\omega}$ for $\omega \in \Omega$. This sends a decorated rooted tree $T$ to a new tree, $\Lambda_{\omega}(T)$ by adding a new root connecting to the root of $T$.  The new root is decorated by $1$ and the edge connecting the new root to the old root is decorated by $\omega$.  This is analogous to taking the $\omega$ integral of the input integral polynomial.  For example:

\[
\begin{tikzpicture}[thick]
\coordinate [label=right: $a$] (A) at (0,0);
\coordinate [label=left: $b$] (B) at (-0.75,1);
\coordinate [label=right: $c$] (C) at (0.75,1);
\coordinate [label=left: $\alpha$] (D) at (-0.35,0.5);
\coordinate [label=right: $\beta$] (E) at (0.35,0.5);
\draw (B) -- (A) -- (C);
\coordinate [label=left: $\Lambda_\omega$\Huge(] (F) at (-1,0.5);
\coordinate [label=right: \Huge)] (G) at (1,0.5);
\coordinate [label=left: $\text{=}$] (H) at (2.25,0.5);
\coordinate [label=right: $a$] (I) at (3.5,0.5);
\coordinate [label=left: $b$] (J) at (2.75,1.5);
\coordinate [label=right: $c$] (K) at (4.25,1.5);
\coordinate [label=right: $1$] (L) at (3.5,-0.5);
\coordinate [label=left: $\alpha$] (M) at (3.15,1);
\coordinate [label=right: $\beta$] (N) at (3.85,1);
\coordinate [label=left: $\omega$] (O) at (3.5,0);
\draw (J) -- (I) -- (K);
\draw (L) -- (I);
\foreach \n in {A,B,C,I,J,K,L} \node at (\n)[circle,fill,inner sep=1.5pt]{};
\end{tikzpicture}
\]
This example is equivalent to applying $\int_\omega$ to the integral $a\left(\int_\alpha b\right)\left(\int_\beta c\right)$, giving us
\[
\int_\omega\left(a\left(\int_\alpha b\right)\left(\int_\beta c\right)\right),
\]
or, in operator notation, applying $P_\omega$ to $aP_\alpha(b)P_\beta(c)$, giving
\[
P_{\omega}\left(aP_{\alpha}(b)P_{\beta}(c)\right).
\]

We say two trees $T_1$ and $T_2$ are {\bf equivalent} in $\mathfrak{E}(\Omega)$, and write $T_1=T_2$, if the integral equations corresponding to $T_1$ and $T_2$ are equivalent equations, that is, they have the same solution set.  Theorem~\ref{theorem:MTRBA} then immediately gives us

\begin{theorem}\label{theorem:RBtrees}
Let $\Omega$ be an indexing set, let $a,f,g \in C(\R)$, and let $\alpha,\beta \in \Omega$.  The following identity is true in $\mathfrak{E}(\Omega)$:

\begin{equation}\label{eq:RBtrees}
\begin{aligned}
\begin{tikzpicture}[thick]
\coordinate [label=right: $a$] (A) at (0,0);
\coordinate [label=left: $f$] (B) at (-0.75,1);
\coordinate [label=right: $g$] (C) at (0.75,1);
\coordinate [label=left: $\alpha$] (D) at (-0.35,0.5);
\coordinate [label=right: $\beta$] (E) at (0.35,0.5);
\draw (B) -- (A) -- (C);
\coordinate [label=left: $\text{=}$] (F) at (2.25,0.5);
\coordinate [label=right: $\tau_\beta a$] (A1) at (3,-0.5);
\coordinate [label=right: $\tau_\beta^{-1} f$] (A2) at (3,0.5);
\coordinate [label=right: $g$] (A3) at (3,1.5);
\coordinate [label=left: $\alpha$] (A4) at (3,0);
\coordinate [label=left: $\beta$] (A5) at (3,1);
\draw (A1) -- (A2) -- (A3);
\coordinate [label=left: $+$] (G) at (5,0.5);
\coordinate [label=right: $\tau_\alpha a$] (B1) at (5.75,-0.5);
\coordinate [label=right: $\tau_\alpha^{-1} g$] (B2) at (5.75,0.5);
\coordinate [label=right: $f$] (B3) at (5.75,1.5);
\coordinate [label=left: $\beta$] (B4) at (5.75,0);
\coordinate [label=left: $\alpha$] (B5) at (5.75,1);
\draw (B1) -- (B2) -- (B3);
\foreach \n in {A,B,C,A1,A2,A3,B1,B2,B3} \node at (\n)[circle,fill,inner sep=1.5pt]{};
\end{tikzpicture}
\end{aligned}
\end{equation}
\end{theorem}

Note that Eq.~\eqref{eq:RBtrees} is simply Eq.~\eqref{MTRBA} written in the form of trees.  Observe that on the left-hand side of this equation the tree contains a branching point, which corresponds to a product of integrals; while the two trees on the right-hand side have no branching but instead are chains of length two, corresponding to two iterated (i.e. operator linear) integrals.  In addition, observe that each tree in the forest on the right-hand side of Eq~\eqref{eq:RBtrees} contains the same number of edges as the tree on the left-hand side.

Eq.~\eqref{eq:RBtrees} can be applied to any two branches emanating from any branching point in a tree.  For example, applying the result to the branching point $b$ and chains ending in $g$ and $e$ in the tree below, we obtain the equivalence of trees

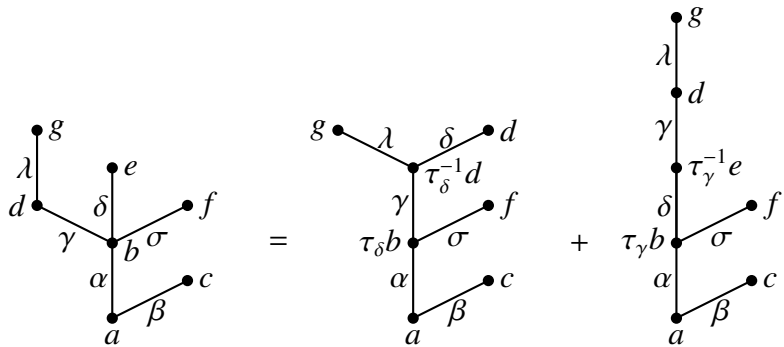
\[
\begin{tikzpicture}[thick]
\coordinate[label=below: $a$] (A1) at (0,0);
\coordinate[label=right: $c$] (A2) at (1,0.5);
\coordinate (A3) at (0,1);
\coordinate[label=right: $f$] (A4) at (1,1.5);
\coordinate[label=right: $e$] (A5) at (0,2);
\coordinate[label=left: $d$] (A6) at (-1,1.5);
\coordinate[label=right: $g$] (A7) at (-1,2.5);
\coordinate[label=right: $b$] (A') at (0,0.9);
\draw (A2) -- (A1) -- (A3) -- (A6) -- (A7);
\draw (A4) -- (A3) -- (A5);
\coordinate[label=below: $\beta$] (A8) at (0.6,0.4);
\coordinate[label=left: $\alpha$] (A9) at (0.1,0.5);
\coordinate[label=below: $\gamma$] (A10) at (-0.6,1.3);
\coordinate[label=below: $\sigma$] (A11) at (0.6,1.3);
\coordinate[label=left: $\lambda$] (A12) at (-0.9,2);
\coordinate[label=left: $\delta$] (A13) at (0.1,1.5);
\foreach \n in {A1,A2,A3,A4,A5,A6,A7} \node at (\n)[circle,fill,inner sep=1.5pt]{};
\coordinate[label=left: \text{=}] (A) at (2.5,1);
\coordinate[label=below: $a$] (B1) at (7.5,0);
\coordinate[label=right: $c$] (B2) at (8.5,0.5);
\coordinate[label=left: $\tau_\gamma b$] (B3) at (7.5,1);
\coordinate[label=right: $f$] (B4) at (8.5,1.5);
\coordinate[label=right: $\tau_\gamma^{-1}e$] (B5) at (7.5,2);
\coordinate[label=right: $d$] (B6) at (7.5,3);
\coordinate[label=right: $g$] (B7) at (7.5,4);
\draw (B2) -- (B1) -- (B3) -- (B6) -- (B7);
\draw (B4) -- (B3) -- (B5);
\coordinate[label=below: $\beta$] (B8) at (8.1,0.4);
\coordinate[label=left: $\alpha$] (B9) at (7.6,0.5);
\coordinate[label=left: $\gamma$] (B10) at (7.6,2.5);
\coordinate[label=below: $\sigma$] (B11) at (8.1,1.3);
\coordinate[label=left: $\lambda$] (B12) at (7.6,3.5);
\coordinate[label=left: $\delta$] (B13) at (7.6,1.5);
\foreach \n in {B1,B2,B3,B4,B5,B6,B7} \node at (\n)[circle,fill,inner sep=1.5pt]{};
\coordinate[label=left: $+$] (B) at (6.5,1);
\coordinate[label=below: $a$] (C1) at (4,0);
\coordinate[label=right: $c$] (C2) at (5,0.5);
\coordinate[label=left: $\tau_\delta b$] (C3) at (4,1);
\coordinate[label=right: $f$] (C4) at (5,1.5);
\coordinate (C5) at (4,2);
\coordinate[label=right: $\tau_\delta^{-1}d$] (C') at (4,1.9);
\coordinate[label=right: $d$] (C6) at (5,2.5);
\coordinate[label=left: $g$] (C7) at (3,2.5);
\draw (C2) -- (C1) -- (C3);
\draw (C4) -- (C3) -- (C5);
\draw (C6) -- (C5) -- (C7);
\coordinate[label=below: $\beta$] (C8) at (4.6,0.4);
\coordinate[label=left: $\alpha$] (C9) at (4.1,0.5);
\coordinate[label=left: $\delta$] (C10) at (4.7,2.4);
\coordinate[label=below: $\sigma$] (C11) at (4.6,1.3);
\coordinate[label=left: $\lambda$] (C12) at (3.9,2.4);
\coordinate[label=left: $\gamma$] (C13) at (4.1,1.5);
\foreach \n in {C1,C2,C3,C4,C5,C6,C7} \node at (\n)[circle,fill,inner sep=1.5pt]{};
\end{tikzpicture}
\]
This translates to the integral polynomial identity
\begin{align*}
aP_\alpha\left(bP_\gamma\left(dP_\lambda(g)\right)P_\delta(e)P_\kappa(f)\right)P_\beta(c) &= aP_\alpha\left(\tau_\delta bP_\gamma\left(\tau_\delta^{-1} dP_\lambda(g)P_\delta(e)\right)P_\kappa(f)\right)P_\beta(c)\\
&\hskip1cm+aP_\alpha\left(\tau_\gamma bP_\delta\left(\tau_\gamma^{-1}eP_\gamma\left(dP_\lambda(g)\right)\right)P_\kappa(f)\right)P_\beta(c).
\end{align*}
Notice how the twisted Rota-Baxter identity in Eq.~\eqref{MTRBA} is applied inside an iterated integral, so when the identity is applied the rest of the terms stay the same.

\section{Reduction Algorithm}

In the following section, we construct an algorithm that reduces a forest in $\mathfrak{E}(\Omega)$ to an equivalent one in that does not contain any branching points.  The two forests are equivalent in the sense that the integral equations associated with them are equivalent equations, that is, they have the same solution set.  As a result, this algorithm transforms any separable Volterra integral equation into an equivalent one in operator-linear form.

Before stating the algorithm, we introduce some terminology that will be used.  A {\bf terminal branch} of a tree $T$ is a chain starting at a branching point of $T$ and ending at a leaf, with no branching points in between.  If there are no branching points in the tree, we say that the entire tree $T$ is a terminal branch.  Given a tree $T$, let $E(T)$ be the total number of edges in $T$, $N(T)$ be the total number of terminal branches of $T$, and $D(T)$ be the sum of the lengths (i.e. the total number of edges) of the terminal branches in $T$.  Note that for any tree $T$, we have $E(T) \ge D(T)$ by definition, and $D(T) \ge N(T)$, since every terminal branch has at least one edge.  If $F$ is a forest, we can similarly define $E(F)$, $N(F)$, and $D(F)$.  

\begin{algorithm}
\caption{Reduction Algorithm}\label{redalg}
\begin{align*}
&\text{{\bf Input:} A vertex-edge decorated rooted forest } F_0 \in \mathfrak{E}(\Omega).\\
&\text{{\bf Output:} A forest } F' \in \mathfrak{E}(\Omega) \text{ that is equivalent to } F_0 \text{ and contains no branching points.}
\end{align*}

\begin{enumerate}
\item Pick any tree $T$ in the forest, and pick any branching point $x$ in $T$ of maximum height.  Observe that we can write $T(x)$ as 

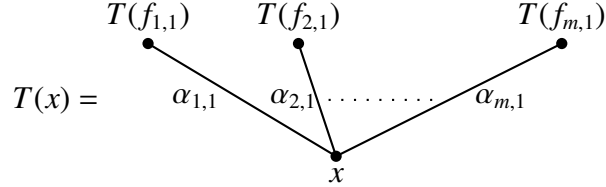
\[
\begin{tikzpicture}[thick]
\coordinate[label=below: $x$] (A) at (0,0);
\coordinate[label=above: $T(f_{1,1})$] (B) at (-2.5,1.5);
\coordinate[label=above: $T(f_{2,1})$] (C) at (-0.5,1.5);
\coordinate[label=above: $T(f_{m,1})$] (D) at (3,1.5);
\draw (B) -- (A) -- (C);
\draw (A) -- (D);
\foreach \n in {A,B,C,D} \node at (\n)[circle,fill,inner sep=1.5pt]{};
\coordinate[label=left: $T(x) \text{ =}$] (E) at (-3,0.75);
\coordinate[label=left: $\alpha_{1,1}$] (F) at (-1.4,0.75);
\coordinate[label=left: $\alpha_{2,1}$] (G) at (-0.1,0.75);
\coordinate[label=right: $\alpha_{m,1}$] (H) at (1.7,0.75);
\draw[loosely dotted] (G) -- (1.4,0.75);
\end{tikzpicture}
\]
where
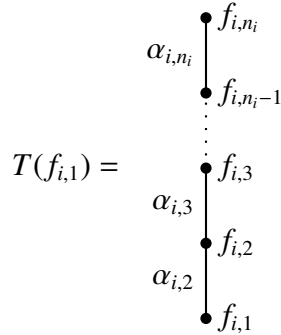
\[
\begin{tikzpicture}[thick]
\coordinate[label=right: $f_{i,1}$] (A) at (0,0);
\coordinate[label=right: $f_{i,2}$] (B) at (0,1);
\coordinate[label=right: $f_{i,3}$] (C) at (0,2);
\coordinate[label=right: $f_{i,n_i-1}$] (D) at (0,3);
\coordinate[label=right: $f_{i,n_i}$] (E) at (0,4);
\draw (A) -- (C);
\draw (D) -- (E);
\draw[loosely dotted] (C) -- (D);
\foreach \n in {A,B,C,D,E} \node at (\n)[circle,fill,inner sep=1.5pt]{};
\coordinate[label=left: $\alpha_{i,2}$] (F) at (0,0.5);
\coordinate[label=left: $\alpha_{i,3}$] (G) at (0,1.5);
\coordinate[label=left: $\alpha_{i,n_i}$] (H) at (0,3.5);
\coordinate[label=left: $T(f_{i,1}) \text{ =}$] (I) at (-1,2);
\end{tikzpicture}
\]

\item Using Theorem~\ref{theorem:RBtrees}, perform the twisted Rota-Baxter identity on the terminal branches above $x$ ending in $f_{1,n_1}$ and $f_{2,n_2}$, producing two new trees from $T(x)$:

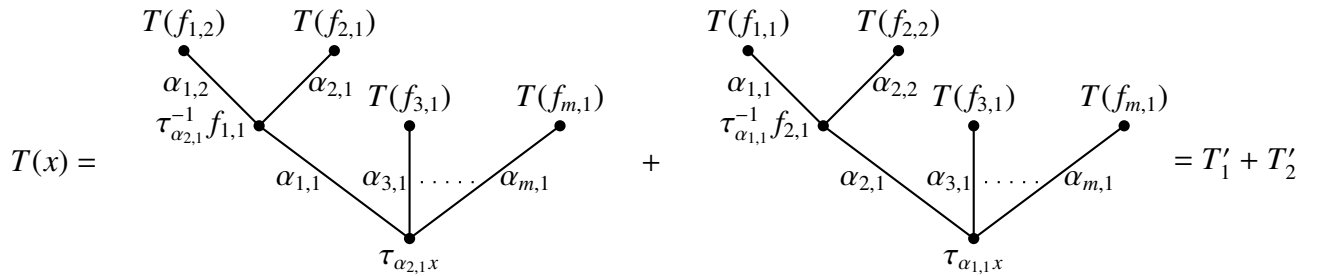
\[
\begin{tikzpicture}[thick]
\coordinate[label=below: $\tau_{\alpha_{2,1}x}$] (A1) at (0,0);
\coordinate[label=left: $\tau_{\alpha_{2,1}}^{-1}f_{1,1}$] (A2) at (-2,1.5);
\coordinate[label=above: $T(f_{3,1})$] (A3) at (0,1.5);
\coordinate[label=above: $T(f_{m,1})$] (A4) at (2,1.5);
\coordinate[label=above: $T(f_{1,2})$] (A5) at (-3,2.5);
\coordinate[label=above: $T(f_{2,1})$] (A6) at (-1,2.5);
\draw (A5) -- (A2) -- (A1) -- (A3);
\draw (A1) -- (A4);
\draw (A6) -- (A2);
\foreach \n in {A1,A2,A3,A4,A5,A6} \node at (\n)[circle,fill,inner sep=1.5pt]{};
\coordinate[label=left: $\alpha_{1,1}$] (A7) at (-1,0.75);
\coordinate[label=left: $\alpha_{3,1}$] (A8) at (0.15,0.75);
\coordinate[label=right: $\alpha_{m,1}$] (A9) at (1.05,0.75);
\coordinate[label=left: $\alpha_{1,2}$] (A10) at (-2.5,2);
\coordinate[label=right: $\alpha_{2,1}$] (A11) at (-1.5,2);
\draw[loosely dotted] (A8) -- (0.95,0.75);
\coordinate[label=left: $T(x) \text{ =}$] (A) at (-4,1);
\coordinate[label=left: $+$] (B) at (3.5,1);
\coordinate[label=below: $\tau_{\alpha_{1,1}x}$] (B1) at (7.5,0);
\coordinate[label=left: $\tau_{\alpha_{1,1}}^{-1}f_{2,1}$] (B2) at (5.5,1.5);
\coordinate[label=above: $T(f_{3,1})$] (B3) at (7.5,1.5);
\coordinate[label=above: $T(f_{m,1})$] (B4) at (9.5,1.5);
\coordinate[label=above: $T(f_{1,1})$] (B5) at (4.5,2.5);
\coordinate[label=above: $T(f_{2,2})$] (B6) at (6.5,2.5);
\draw (B5) -- (B2) -- (B1) -- (B3);
\draw (B1) -- (B4);
\draw (B6) -- (B2);
\foreach \n in {B1,B2,B3,B4,B5,B6} \node at (\n)[circle,fill,inner sep=1.5pt]{};
\coordinate[label=left: $\alpha_{2,1}$] (B7) at (6.5,0.75);
\coordinate[label=left: $\alpha_{3,1}$] (B8) at (7.65,0.75);
\coordinate[label=right: $\alpha_{m,1}$] (B9) at (8.55,0.75);
\coordinate[label=left: $\alpha_{1,1}$] (B10) at (5,2);
\coordinate[label=right: $\alpha_{2,2}$] (B11) at (6,2);
\draw[loosely dotted] (B8) -- (8.45,0.75);
\coordinate[label=right: $\text{= } T_1' + T_2'$] (C) at (10,1);
\end{tikzpicture}
\]

\item Form two new trees $T_1$ and $T_2$ from $T$.  Each new tree is $B(x)$ adjoined by $T_1'$ or $T_2'$ in place of the original $T(x)$.  Remove $T$ from the forest and add the two new trees $T_1$ and $T_2$ to the forest in its place.

\item If there are no branching points in any tree in the forest, the algorithm terminates.  If not, go back to step (1) and repeat.
\end{enumerate}
\end{algorithm}

\begin{theorem}\label{theorem:algorithm}
Algorithm~\ref{redalg} terminates after finitely many steps.  The output forest $F'$ contains no branching points and is equivalent to the input forest $F_0$.
\end{theorem}
\begin{proof}
By Theorem~\ref{theorem:RBtrees}, the forest produced after each iteration of the algorithm is equivalent to the previous forest, so assuming the algorithm terminates, the input and output forests will be equivalent.  Also, it is clear that if the algorithm terminates, the output forest will have no branching points.  Thus it remains to show that the algorithm terminates after finitely many steps.  Note that the algorithm terminates when $N(T)=1$ for all trees $T$ in the forest.

At each step in the algorithm, a tree $T$ is replaced by two new trees $T_1$ and $T_2$.  Observe by Theorem~\ref{theorem:RBtrees}, $E(T_1) = E(T_2) = E(T)$, as the twisted Rota-Baxter identity in Eq.~\eqref{eq:RBtrees} does not change the total number of edges in each tree of the forest.  Thus at every step of the algorithm, every tree $T$ in the forest satisfies $E(T) \le E(F_0)$, where $F_0$ is the input forest.  Since $N(T) \le D(T) \le E(T)$ for a tree $T$, we also have $D(T) \le E(F_0)$ and $N(T) \le E(F_0)$ for every tree $T$ in every step of the algorithm.

We claim that when tree $T$ is replaced by trees $T_1$ and $T_2$ in the algorithm, each $T_i$ ($i=1,2$) satisfies either
\begin{itemize}
\item[(a)] $N(T_i) < N(T)$, or
\item[(b)] $N(T_i) = N(T) \text{ and } D(T_i) < D(T)$.
\end{itemize}
Since $D(T_i)>0$, the second case can only occur finitely many times until the first case occurs.  Again, since $N(T_i) > 0$, this can also only occur finitely many times until eventually $N(T_i) = 1$.  Note that in case (a), it is possible to have $D(T_i) > D(T)$; however, since $D(T) < E(F_0)$ for every tree in the every iteration of the algorithm, there is a fixed maximum value for $D(T_i)$.

Consider an iteration of the algorithm in which we are performing the identity in Eq.~\eqref{eq:RBtrees} at branching point $x$ in tree $T$ to the branches ending in $f_{1,n_1}$ and $f_{2,n_2}$.  Let $n_i$ be the length of terminal branch $i$ with root $x$.  Observe that if $n_1>1$, then in $T_1$, the first edge of the terminal branch of $T$ ending in $f_{1,n_1}$ now occurs below the new branching point $\tau_{\alpha_{2,1}}^{-1}f_{1,1}$ (which is indeed a branching point in $T_1$ since $n_1 > 1$).  Thus the terminal branch of $T_1$ ending in $f_{1,n_1}$ has one fewer edge than it did in $T$.  If $n_1=1$, then $\tau_{\alpha_{2,1}}^{-1}f_{1,1}$ is not a branching point of $T_1$, as there is no $f_{1,2}$ to form a second branch.  In either case, the terminal branch of $T_1$ ending in $f_{2,n_2}$ has the same number of edges as it did in $T$.  We can make similar observations regarding $T_2$ and $n_2$.  Also, note that none of the other terminal branches emanating from $x$ in $T$ (if they exist) are affected by Eq.~\eqref{eq:RBtrees}, and are thus still terminal branches in $T_1$ and $T_2$ emanating from $\tau_{\alpha_{2,1}}x$ and $\tau_{\alpha_{1,1}}x$, respectively, of the same length as they are in $T$.

Summarizing, we see that for $i=1,2$, we have 
\[
\begin{cases}\text{ if } n_i=1, \text{ then } N(T_i) = N(T)-1, \\ \text{ if } n_i>1, \text{ then } N(T_i)=N(T) \text{ and } D(T_i)=D(T)-1.  \end{cases}
\]
As a result, for every iteration of the algorithm, either (a) or (b) above will occur for each output tree $T_1$ and $T_2$, meaning that the algorithm will terminate.
\end{proof}


\section{Example of Operator Linear Reduction}

In this section we apply Algorithm~\ref{redalg} to several examples, and analyze the corresponding integral equation reductions.  

\begin{example}\label{iterated2}
In this example we will work through each iteration of the algorithm, stating both the resulting forest and its equivalent integral polynomial.

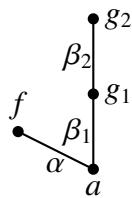
\begin{multicols}{2}
Input:
\[
\begin{tikzpicture}[thick]
\coordinate[label=below: $a$] (A) at (0,0);
\coordinate[label=right: $g_1$] (B) at (0,1);
\coordinate[label=right: $g_2$] (C) at (0,2);
\coordinate[label=above: $f$] (D) at (-1,0.5);
\draw (D) -- (A) -- (B) -- (C);
\foreach \n in {A,B,C,D} \node at (\n)[circle,fill,inner sep=1.5pt]{};
\coordinate[label=left: $\beta_1$] (E) at (0.15,0.5);
\coordinate[label=left: $\beta_2$] (F) at (0.15,1.5);
\coordinate[label=below: $\alpha$] (G) at (-0.5,0.3);
\end{tikzpicture}
\]

Integral Representation:
\[
a\left(\int_\alpha f\right)\left(\int_{\beta_1}g_1\left(\int_{\beta_2}g_2\right)\right)
\]
\end{multicols}

\begin{multicols}{2}
Iteration 1:
\[
\begin{tikzpicture}[thick]
\coordinate[label=below: $\tau_{\beta_1}a$] (A1) at (0,0);
\coordinate[label=right: $\tau_{\beta_1}^{-1}f$] (A2) at (0,1);
\coordinate[label=right: $g_1$] (A3) at (0,2);
\coordinate[label=right: $g_2$] (A4) at (0,3);
\draw (A1) -- (A4);
\foreach \n in {A1,A2,A3,A4} \node at (\n)[circle,fill,inner sep=1.5pt]{};
\coordinate[label=left: $\alpha$] (A5) at (0.1,0.5);
\coordinate[label=left: $\beta_1$] (A6) at (0.15,1.5);
\coordinate[label=left: $\beta_2$] (A7) at (0.15,2.5);
\coordinate[label=left: $+$] (A) at (2,1.5);
\coordinate[label=below: $\tau_\alpha a$] (B1) at (4,0.5);
\coordinate[label=right: $\tau_{\alpha}^{-1}g_1$] (B2) at (4,1.5);
\coordinate[label=above: $f$] (B3) at (3,2.5);
\coordinate[label=above: $g_2$] (B4) at (5,2.5);
\draw (B3) -- (B2) -- (B1);
\draw (B4) -- (B2);
\foreach \n in {B1,B2,B3,B4} \node at (\n)[circle,fill,inner sep=1.5pt]{};
\coordinate[label=left: $\beta_1$] (B5) at (4.15,1);
\coordinate[label=above: $\alpha$] (B5) at (3.5,2);
\coordinate[label=above: $\beta_2$] (B6) at (4.5,2);
\end{tikzpicture}
\]

Integral Representation:
\begin{multline*}
\tau_{\beta_1}a\int_\alpha\tau_{\beta_1}^{-1}f\left(\int_{\beta_1}g_1\left(\int_{\beta_2}g_2\right)\right)\\
+\tau_\alpha a\left(\int_{\beta_1}\tau_\alpha^{-1}g_1\left(\int_\alpha f\right)\right)\left(\int_{\beta_2}g_2\right)
\end{multline*}
\end{multicols}

\begin{multicols}{2}
Iteration 2:
\[
\begin{tikzpicture}[thick]
\coordinate[label=below: $\tau_{\beta_1}a$] (A1) at (0,0);
\coordinate[label=right: $\tau_{\beta_1}^{-1}f$] (A2) at (0,1);
\coordinate[label=right: $g_1$] (A3) at (0,2);
\coordinate[label=right: $g_2$] (A4) at (0,3);
\draw (A1) -- (A4);
\foreach \n in {A1,A2,A3,A4} \node at (\n)[circle,fill,inner sep=1.5pt]{};
\coordinate[label=left: $\alpha$] (A5) at (0.1,0.5);
\coordinate[label=left: $\beta_1$] (A6) at (0.15,1.5);
\coordinate[label=left: $\beta_2$] (A7) at (0.15,2.5);
\coordinate[label=left: $+$] (A) at (2,1.5);
\coordinate[label=below: $\tau_{\alpha}a$] (B1) at (3,0);
\coordinate[label=right: $\tau_{\alpha}^{-1}g_1\tau_{\beta_2}$] (B2) at (3,1);
\coordinate[label=right: $\tau_{\beta_2}^{-1}f$] (B3) at (3,2);
\coordinate[label=right: $g_2$] (B4) at (3,3);
\draw (B1) -- (B4);
\foreach \n in {B1,B2,B3,B4} \node at (\n)[circle,fill,inner sep=1.5pt]{};
\coordinate[label=left: $\beta_1$] (B5) at (3.15,0.5);
\coordinate[label=left: $\alpha$] (B6) at (3.1,1.5);
\coordinate[label=left: $\beta_2$] (B7) at (3.15,2.5);
\coordinate[label=left: $+$] (B) at (5,1.5);
\coordinate[label=below: $\tau_{\alpha}a$] (C1) at (6,0);
\coordinate[label=right: $g_1$] (C2) at (6,1);
\coordinate[label=right: $\tau_{\alpha}^{-1}g_2$] (C3) at (6,2);
\coordinate[label=right: $f$] (C4) at (6,3);
\draw (C1) -- (C4);
\foreach \n in {C1,C2,C3,C4} \node at (\n)[circle,fill,inner sep=1.5pt]{};
\coordinate[label=left: $\beta_1$] (C5) at (6.15,0.5);
\coordinate[label=left: $\beta_2$] (C6) at (6.15,1.5);
\coordinate[label=left: $\alpha$] (C7) at (6.1,2.5);
\end{tikzpicture}
\]

Integral Representation:
\begin{multline*}
\tau_{\beta_1}a\int_\alpha\tau_{\beta_1}^{-1}f\left(\int_{\beta_1}g_1\left(\int_{\beta_2}g_2\right)\right)\\
+\tau_\alpha a\int_{\beta_1}\tau_\alpha^{-1}g_1\tau_{\beta_2}\left(\int_\alpha \tau_{\beta_2}^{-1}f\left(\int_{\beta_2}g_2\right)\right)\\
+\tau_\alpha a\int_{\beta_1}g_1\left(\int_{\beta_2}\tau_\alpha^{-1}g_2\left(\int_\alpha f\right)\right)
\end{multline*}
\end{multicols}
This example allows us to see how $N(T)$ and $D(T)$ change in each iteration of Algorithm~\ref{redalg}.  Consider the first iteration, where (in the notation of the algorithm) we have
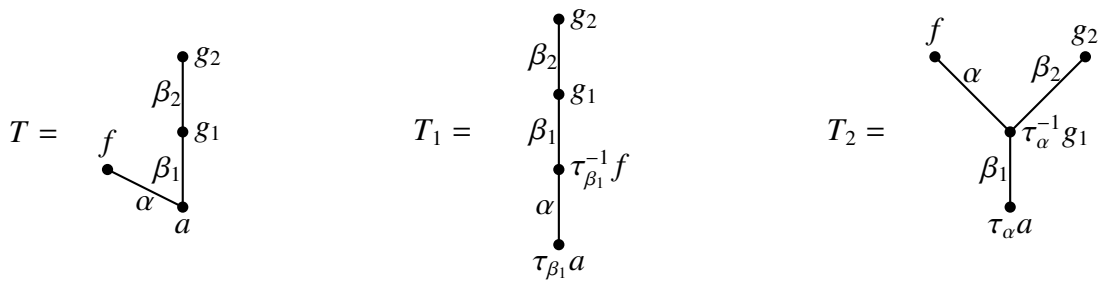
\[
\begin{tikzpicture}[thick]
\coordinate[label=left: $T \text{ =}$] (H) at (-1.5,1.5);
\coordinate[label=below: $a$] (A) at (0,0.5);
\coordinate[label=right: $g_1$] (B) at (0,1.5);
\coordinate[label=right: $g_2$] (C) at (0,2.5);
\coordinate[label=above: $f$] (D) at (-1,1);
\draw (D) -- (A) -- (B) -- (C);
\foreach \n in {A,B,C,D} \node at (\n)[circle,fill,inner sep=1.5pt]{};
\coordinate[label=left: $\beta_1$] (E) at (0.15,1);
\coordinate[label=left: $\beta_2$] (F) at (0.15,2);
\coordinate[label=below: $\alpha$] (G) at (-0.5,0.8);
\coordinate[label=left: $T_1 \text{ =}$] (A0) at (4,1.5);
\coordinate[label=below: $\tau_{\beta_1}a$] (A1) at (5,0);
\coordinate[label=right: $\tau_{\beta_1}^{-1}f$] (A2) at (5,1);
\coordinate[label=right: $g_1$] (A3) at (5,2);
\coordinate[label=right: $g_2$] (A4) at (5,3);
\draw (A1) -- (A4);
\foreach \n in {A1,A2,A3,A4} \node at (\n)[circle,fill,inner sep=1.5pt]{};
\coordinate[label=left: $\alpha$] (A5) at (5.1,0.5);
\coordinate[label=left: $\beta_1$] (A6) at (5.15,1.5);
\coordinate[label=left: $\beta_2$] (A7) at (5.15,2.5);
\coordinate[label=left: $T_2 \text{ =}$] (AP) at (9.5,1.5);
\coordinate[label=below: $\tau_\alpha a$] (B1) at (11,0.5);
\coordinate[label=right: $\tau_{\alpha}^{-1}g_1$] (B2) at (11,1.5);
\coordinate[label=above: $f$] (B3) at (10,2.5);
\coordinate[label=above: $g_2$] (B4) at (12,2.5);
\draw (B3) -- (B2) -- (B1);
\draw (B4) -- (B2);
\foreach \n in {B1,B2,B3,B4} \node at (\n)[circle,fill,inner sep=1.5pt]{};
\coordinate[label=left: $\beta_1$] (B5) at (11.15,1);
\coordinate[label=above: $\alpha$] (B5) at (10.5,2);
\coordinate[label=above: $\beta_2$] (B6) at (11.5,2);
\end{tikzpicture}
\]
Here $N(T)=2$ and $D(T) = 3$.  Since the terminal branch of $T$ ending in $f$ has length 1, that branch is no longer present in $T_1$, so we have $N(T_1) = 1 = N(T)-1$.  Since the terminal branch of $T$ ending in $g_2$ has length 2, that branch has one fewer edge in $T_2$, so we have $N(T_2) = 2 = N(T)$, while $D(T_2) = 2 = D(T)-1$.

\end{example}

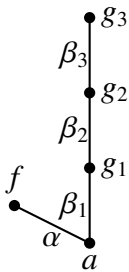
\begin{example}\label{iterated3}
$\,$

Input:
\[
\begin{tikzpicture}[thick]
\coordinate[label=below: $a$] (A) at (0,0);
\coordinate[label=right: $g_1$] (B) at (0,1);
\coordinate[label=right: $g_2$] (C) at (0,2);
\coordinate[label=right: $g_3$](D) at (0,3);
\coordinate[label=above: $f$] (E) at (-1,0.5);
\draw (E) -- (A) -- (B) -- (C) -- (D);
\foreach \n in {A,B,C,D,E} \node at (\n)[circle,fill,inner sep=1.5pt]{};
\coordinate[label=left: $\beta_1$] (E) at (0.15,0.5);
\coordinate[label=left: $\beta_2$] (F) at (0.15,1.5);
\coordinate[label=left: $\beta_3$] (G) at (0.15,2.5);
\coordinate[label=below: $\alpha$] (H) at (-0.5,0.3);
\end{tikzpicture}
\]

Iteration 1:
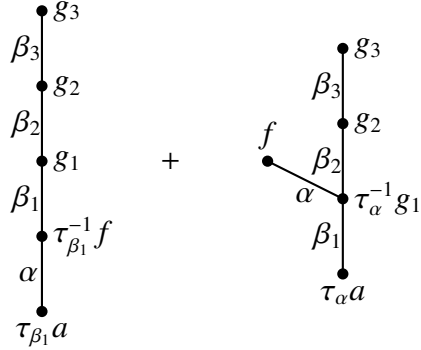
\[
\begin{tikzpicture}[thick]
\coordinate[label=below: $\tau_{\beta_1}a$] (A1) at (0,0);
\coordinate[label=right: $\tau_{\beta_1}^{-1}f$] (A2) at (0,1);
\coordinate[label=right: $g_1$] (A3) at (0,2);
\coordinate[label=right: $g_2$](A4) at (0,3);
\coordinate[label=right: $g_3$] (A5) at (0,4);
\draw (A1) -- (A5);
\foreach \n in {A1,A2,A3,A4,A5} \node at (\n)[circle,fill,inner sep=1.5pt]{};
\coordinate[label=left: $\alpha$] (A6) at (0.1,0.5);
\coordinate[label=left: $\beta_1$] (A7) at (0.15,1.5);
\coordinate[label=left: $\beta_2$] (A8) at (0.15,2.5);
\coordinate[label=left: $\beta_3$] (A9) at (0.15,3.5);
\coordinate[label=left: $+$] (A) at (2,2);
\coordinate[label=below: $\tau_\alpha a$] (B1) at (4,0.5);
\coordinate[label=right: $\tau_\alpha^{-1}g_1$] (B2) at (4,1.5);
\coordinate[label=right: $g_2$] (B3) at (4,2.5);
\coordinate[label=right: $g_3$](B4) at (4,3.5);
\coordinate[label=above: $f$] (B5) at (3,2);
\draw (B1) -- (B4);
\draw (B2) -- (B5);
\foreach \n in {B1,B2,B3,B4,B5} \node at (\n)[circle,fill,inner sep=1.5pt]{};
\coordinate[label=left: $\beta_1$] (B6) at (4.15,1);
\coordinate[label=left: $\beta_2$] (B7) at (4.15,2);
\coordinate[label=left: $\beta_3$] (B8) at (4.15,3);
\coordinate[label=below: $\alpha$] (B9) at (3.5,1.8);
\end{tikzpicture}
\]

Iteration 2:
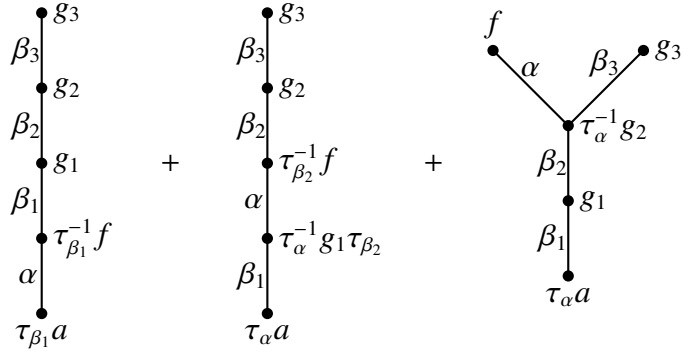
\[
\begin{tikzpicture}[thick]
\coordinate[label=below: $\tau_{\beta_1}a$] (A1) at (0,0);
\coordinate[label=right: $\tau_{\beta_1}^{-1}f$] (A2) at (0,1);
\coordinate[label=right: $g_1$] (A3) at (0,2);
\coordinate[label=right: $g_2$](A4) at (0,3);
\coordinate[label=right: $g_3$] (A5) at (0,4);
\draw (A1) -- (A5);
\foreach \n in {A1,A2,A3,A4,A5} \node at (\n)[circle,fill,inner sep=1.5pt]{};
\coordinate[label=left: $\alpha$] (A6) at (0.1,0.5);
\coordinate[label=left: $\beta_1$] (A7) at (0.15,1.5);
\coordinate[label=left: $\beta_2$] (A8) at (0.15,2.5);
\coordinate[label=left: $\beta_3$] (A9) at (0.15,3.5);
\coordinate[label=left: $+$] (A) at (2,2);
\coordinate[label=below: $\tau_{\alpha}a$] (B1) at (3,0);
\coordinate[label=right: $\tau_{\alpha}^{-1}g_1\tau_{\beta_2}$] (B2) at (3,1);
\coordinate[label=right: $\tau_{\beta_2}^{-1}f$] (B3) at (3,2);
\coordinate[label=right: $g_2$](B4) at (3,3);
\coordinate[label=right: $g_3$] (B5) at (3,4);
\draw (B1) -- (B5);
\foreach \n in {B1,B2,B3,B4,B5} \node at (\n)[circle,fill,inner sep=1.5pt]{};
\coordinate[label=left: $\beta_1$] (B6) at (3.15,0.5);
\coordinate[label=left: $\alpha$] (B7) at (3.1,1.5);
\coordinate[label=left: $\beta_2$] (B8) at (3.15,2.5);
\coordinate[label=left: $\beta_3$] (B9) at (3.15,3.5);
\coordinate[label=left: $+$] (B) at (5.5,2);
\coordinate[label=below: $\tau_\alpha a$] (C1) at (7,0.5);
\coordinate[label=right: $g_1$] (C2) at (7,1.5);
\coordinate[label=right: $\tau_\alpha^{-1}g_2$] (C3) at (7,2.5);
\coordinate[label=right: $g_3$](C4) at (8,3.5);
\coordinate[label=above: $f$] (C5) at (6,3.5);
\draw (C1) --(C2) -- (C3) -- (C4);
\draw (C3) -- (C5);
\foreach \n in {C1,C2,C3,C4,C5} \node at (\n)[circle,fill,inner sep=1.5pt]{};
\coordinate[label=left: $\beta_1$] (C6) at (7.15,1);
\coordinate[label=left: $\beta_2$] (C7) at (7.15,2);
\coordinate[label=above: $\beta_3$] (C8) at (7.5,3);
\coordinate[label=above: $\alpha$] (C9) at (6.5,3);
\end{tikzpicture}
\]

Iteration 3:
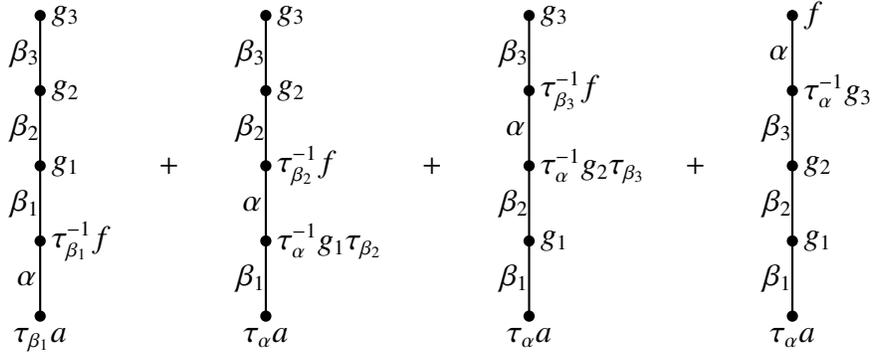
\[
\begin{tikzpicture}[thick]
\coordinate[label=below: $\tau_{\beta_1}a$] (A1) at (0,0);
\coordinate[label=right: $\tau_{\beta_1}^{-1}f$] (A2) at (0,1);
\coordinate[label=right: $g_1$] (A3) at (0,2);
\coordinate[label=right: $g_2$](A4) at (0,3);
\coordinate[label=right: $g_3$] (A5) at (0,4);
\draw (A1) -- (A5);
\foreach \n in {A1,A2,A3,A4,A5} \node at (\n)[circle,fill,inner sep=1.5pt]{};
\coordinate[label=left: $\alpha$] (A6) at (0.1,0.5);
\coordinate[label=left: $\beta_1$] (A7) at (0.15,1.5);
\coordinate[label=left: $\beta_2$] (A8) at (0.15,2.5);
\coordinate[label=left: $\beta_3$] (A9) at (0.15,3.5);
\coordinate[label=left: $+$] (A) at (2,2);
\coordinate[label=below: $\tau_{\alpha}a$] (B1) at (3,0);
\coordinate[label=right: $\tau_{\alpha}^{-1}g_1\tau_{\beta_2}$] (B2) at (3,1);
\coordinate[label=right: $\tau_{\beta_2}^{-1}f$] (B3) at (3,2);
\coordinate[label=right: $g_2$](B4) at (3,3);
\coordinate[label=right: $g_3$] (B5) at (3,4);
\draw (B1) -- (B5);
\foreach \n in {B1,B2,B3,B4,B5} \node at (\n)[circle,fill,inner sep=1.5pt]{};
\coordinate[label=left: $\beta_1$] (B6) at (3.15,0.5);
\coordinate[label=left: $\alpha$] (B7) at (3.1,1.5);
\coordinate[label=left: $\beta_2$] (B8) at (3.15,2.5);
\coordinate[label=left: $\beta_3$] (B9) at (3.15,3.5);
\coordinate[label=left: $+$] (B) at (5.5,2);
\coordinate[label=below: $\tau_{\alpha}a$] (C1) at (6.5,0);
\coordinate[label=right: $g_1$] (C2) at (6.5,1);
\coordinate[label=right: $\tau_\alpha^{-1}g_2\tau_{\beta_3}$] (C3) at (6.5,2);
\coordinate[label=right: $\tau_{\beta_3}^{-1}f$](C4) at (6.5,3);
\coordinate[label=right: $g_3$] (C5) at (6.5,4);
\draw (C1) -- (C5);
\foreach \n in {C1,C2,C3,C4,C5} \node at (\n)[circle,fill,inner sep=1.5pt]{};
\coordinate[label=left: $\beta_1$] (C6) at (6.65,0.5);
\coordinate[label=left: $\beta_2$] (C7) at (6.65,1.5);
\coordinate[label=left: $\alpha$] (C8) at (6.6,2.5);
\coordinate[label=left: $\beta_3$] (C9) at (6.65,3.5);
\coordinate[label=left: $+$] (C) at (9,2);
\coordinate[label=below: $\tau_{\alpha}a$] (D1) at (10,0);
\coordinate[label=right: $g_1$] (D2) at (10,1);
\coordinate[label=right: $g_2$] (D3) at (10,2);
\coordinate[label=right: $\tau_\alpha^{-1}g_3$](D4) at (10,3);
\coordinate[label=right: $f$] (D5) at (10,4);
\draw (D1) -- (D5);
\foreach \n in {D1,D2,D3,D4,D5} \node at (\n)[circle,fill,inner sep=1.5pt]{};
\coordinate[label=left: $\beta_1$] (D6) at (10.15,0.5);
\coordinate[label=left: $\beta_2$] (D7) at (10.15,1.5);
\coordinate[label=left: $\beta_3$] (D8) at (10.15,2.5);
\coordinate[label=left: $\alpha$] (D9) at (10.1,3.5);
\end{tikzpicture}
\]
Using the equivalence of forests to integral polynomials, this example shows that
\begin{multline*}
a\left(\int_\alpha f\right)\left(\int_{\beta_1}g_1\left(\int_{\beta_2}g_2\left(\int_{\beta_3}g_3\right)\right)\right) \\
= \tau_{\beta_1}a\int_\alpha\tau_{\beta_1}^{-1}f\left(\int_{\beta_1} g_1\left(\int_{\beta_2}g_2\left(\int_{\beta_3}g_3\right)\right)\right) + \tau_\alpha a\int_{\beta_1}\tau_\alpha^{-1}g_1\tau_{\beta_2}\left(\int_\alpha \tau_{\beta_2}^{-1}f\left(\int_{\beta_2}g_2\left(\int_{\beta_3}g_3\right)\right)\right)\\
+ \tau_\alpha a \int_{\beta_1}g_1\left(\int_{\beta_2}\tau_\alpha^{-1}g_2\tau_{\beta_3}\left(\int_\alpha \tau_{\beta_3}^{-1}f\left(\int_{\beta_3}g_3\right)\right)\right) + \tau_\alpha \int_{\beta_1}g_1\left(\int_{\beta_2}g_2\left(\int_{\beta_3}\tau_\alpha^{-1}g_3\left(\int_\alpha f\right)\right)\right).
\end{multline*}
\end{example}

\bigskip
We can sometimes use previously computed reductions to help simplify Algorithm~\ref{redalg} on more complicated inputs.
\begin{example}
In this example, we apply Algorithm~\ref{redalg} to a tree containing a branching point with more than two branches.

\bigskip
Input:
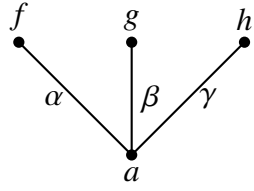
\[
\begin{tikzpicture}[thick]
\coordinate[label=below: $a$] (A) at (0,0);
\coordinate[label=above: $f$] (B) at (-1.5,1.5);
\coordinate[label=above: $g$] (C) at (0,1.5);
\coordinate[label=above: $h$] (D) at (1.5,1.5);
\draw (B) -- (A) -- (D);
\draw (C) -- (A);
\foreach \n in {A,B,C,D} \node at (\n)[circle,fill,inner sep=1.5pt]{};
\coordinate[label=left: $\alpha$] (E) at (-0.75,0.75);
\coordinate[label=right: $\beta$] (F) at (0,0.75);
\coordinate[label=right: $\gamma$] (G) at (0.75,0.75);
\end{tikzpicture}
\]

Iteration 1:
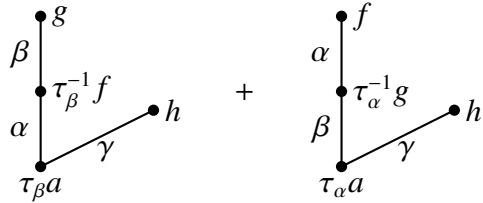
\[
\begin{tikzpicture}[thick]
\coordinate[label=below: $\tau_\beta a$] (A1) at (0,0);
\coordinate[label=right: $\tau_\beta^{-1}f$] (A2) at (0,1);
\coordinate[label=right: $g$] (A3) at (0,2);
\coordinate[label=right: $h$] (A4) at (1.5,0.75);
\draw (A4) -- (A1) -- (A2) -- (A3);
\foreach \n in {A1,A2,A3,A4} \node at (\n)[circle,fill,inner sep=1.5pt]{};
\coordinate[label=left: $\alpha$] (A5) at (0,0.5);
\coordinate[label=left: $\beta$] (A6) at (0,1.5);
\coordinate[label=right: $\gamma$] (A7) at (0.6,0.2);
\coordinate[label=left: $+$] (A) at (3,1);
\coordinate[label=below: $\tau_\alpha a$] (B1) at (4,0);
\coordinate[label=right: $\tau_\alpha^{-1}g$] (B2) at (4,1);
\coordinate[label=right: $f$] (B3) at (4,2);
\coordinate[label=right: $h$] (B4) at (5.5,0.75);
\draw (B4) -- (B1) -- (B2) -- (B3);
\foreach \n in {B1,B2,B3,B4} \node at (\n)[circle,fill,inner sep=1.5pt]{};
\coordinate[label=left: $\beta$] (B5) at (4,0.5);
\coordinate[label=left: $\alpha$] (B6) at (4,1.5);
\coordinate[label=right: $\gamma$] (B7) at (4.6,0.2);
\end{tikzpicture}
\]

Output:
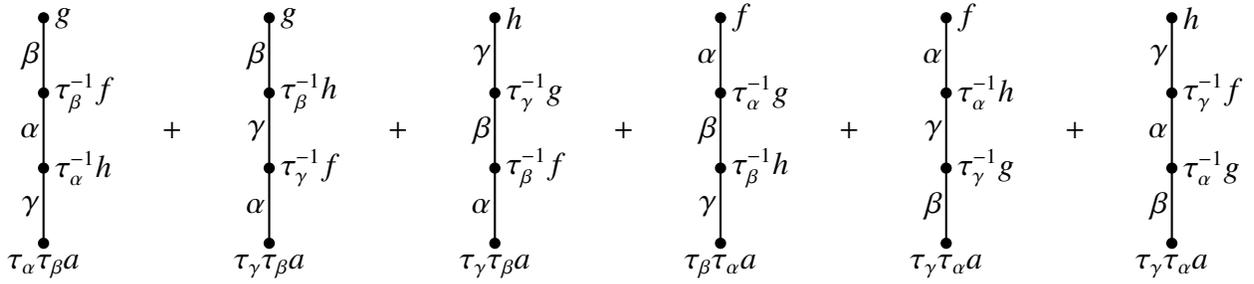
\[
\begin{tikzpicture}[thick]
\coordinate[label=below: $\tau_{\alpha}\tau_\beta a$] (A1) at (0,0);
\coordinate[label=right: $\tau_{\alpha}^{-1}h$] (A2) at (0,1);
\coordinate[label=right: $\tau_\beta^{-1}f$] (A3) at (0,2);
\coordinate[label=right: $g$] (A4) at (0,3);
\draw (A1) -- (A4);
\foreach \n in {A1,A2,A3,A4} \node at (\n)[circle,fill,inner sep=1.5pt]{};
\coordinate[label=left: $\gamma$] (A5) at (0.1,0.5);
\coordinate[label=left: $\alpha$] (A6) at (0.1,1.5);
\coordinate[label=left: $\beta$] (A7) at (0.1,2.5);
\coordinate[label=left: $+$] (A) at (2,1.5);
\coordinate[label=below: $\tau_\gamma\tau_\beta a$] (B1) at (3,0);
\coordinate[label=right: $\tau_{\gamma}^{-1}f$] (B2) at (3,1);
\coordinate[label=right: $\tau_{\beta}^{-1}h$] (B3) at (3,2);
\coordinate[label=right: $g$] (B4) at (3,3);
\draw (B1) -- (B4);
\foreach \n in {B1,B2,B3,B4} \node at (\n)[circle,fill,inner sep=1.5pt]{};
\coordinate[label=left: $\alpha$] (B5) at (3.1,0.5);
\coordinate[label=left: $\gamma$] (B6) at (3.1,1.5);
\coordinate[label=left: $\beta$] (B7) at (3.1,2.5);
\coordinate[label=left: $+$] (B) at (5,1.5);
\coordinate[label=below: $\tau_\gamma\tau_\beta a$] (C1) at (6,0);
\coordinate[label=right: $\tau_\beta^{-1}f$] (C2) at (6,1);
\coordinate[label=right: $\tau_\gamma^{-1}g$] (C3) at (6,2);
\coordinate[label=right: $h$] (C4) at (6,3);
\draw (C1) -- (C4);
\foreach \n in {C1,C2,C3,C4} \node at (\n)[circle,fill,inner sep=1.5pt]{};
\coordinate[label=left: $\alpha$] (C5) at (6.1,0.5);
\coordinate[label=left: $\beta$] (C6) at (6.1,1.5);
\coordinate[label=left: $\gamma$] (C7) at (6.1,2.5);
\coordinate[label=left: $+$] (C) at (8,1.5);
\coordinate[label=below: $\tau_\beta\tau_\alpha a$] (D1) at (9,0);
\coordinate[label=right: $\tau_\beta^{-1}h$] (D2) at (9,1);
\coordinate[label=right: $\tau_\alpha^{-1}g$] (D3) at (9,2);
\coordinate[label=right: $f$] (D4) at (9,3);
\draw (D1) -- (D4);
\foreach \n in {D1,D2,D3,D4} \node at (\n)[circle,fill,inner sep=1.5pt]{};
\coordinate[label=left: $\gamma$] (D5) at (9.1,0.5);
\coordinate[label=left: $\beta$] (D6) at (9.1,1.5);
\coordinate[label=left: $\alpha$] (D7) at (9.1,2.5);
\coordinate[label=left: $+$] (D) at (11,1.5);
\coordinate[label=below: $\tau_\gamma\tau_\alpha a$] (E1) at (12,0);
\coordinate[label=right: $\tau_{\gamma}^{-1}g$] (E2) at (12,1);
\coordinate[label=right: $\tau_\alpha^{-1}h$] (E3) at (12,2);
\coordinate[label=right: $f$] (E4) at (12,3);
\draw (E1) -- (E4);
\foreach \n in {E1,E2,E3,E4} \node at (\n)[circle,fill,inner sep=1.5pt]{};
\coordinate[label=left: $\beta$] (E5) at (12.1,0.5);
\coordinate[label=left: $\gamma$] (E6) at (12.1,1.5);
\coordinate[label=left: $\alpha$] (E7) at (12.1,2.5);
\coordinate[label=left: $+$] (E) at (14,1.5);
\coordinate[label=below: $\tau_\gamma\tau_\alpha a$] (F1) at (15,0);
\coordinate[label=right: $\tau_\alpha^{-1}g$] (F2) at (15,1);
\coordinate[label=right: $\tau_\gamma^{-1}f$] (F3) at (15,2);
\coordinate[label=right: $h$] (F4) at (15,3);
\draw (F1) -- (F4);
\foreach \n in {F1,F2,F3,F4} \node at (\n)[circle,fill,inner sep=1.5pt]{};
\coordinate[label=left: $\beta$] (F5) at (15.1,0.5);
\coordinate[label=left: $\alpha$] (F6) at (15.1,1.5);
\coordinate[label=left: $\gamma$] (F7) at (15.1,2.5);
\end{tikzpicture}
\]
In this example, instead of working through each iteration of Algorithm~\ref{redalg} (which would have taken an additional four iterations), following the first iteration of the algorithm we applied the results of Example~\ref{iterated2} to complete the reduction.  In the language of integrals, this example says
\begin{multline*}
a\left(\int_\alpha f\right)\left(\int_\beta g\right)\left(\int_\gamma h\right) = \tau_\alpha\tau_\beta a\int_\gamma \tau_\alpha^{-1}h\left(\int_\alpha \tau_\beta^{-1}f\left(\int_\beta g\right)\right) + \tau_\gamma\tau_\beta a\int_\alpha \tau_\gamma^{-1}f\left(\int_\gamma\tau_\beta^{-1}h\left(\int_\beta g\right)\right) \\
+ \tau_\gamma\tau_\beta a\int_\alpha \tau_\beta^{-1}f\left(\int_\beta\tau_\gamma^{-1}g\left(\int_\gamma h\right)\right)+ \tau_\beta\tau_\alpha a\int_\gamma \tau_\beta^{-1}h\left(\int_\beta \tau_\alpha^{-1}g\left(\int_\alpha f\right)\right)\\
+\tau_\gamma \tau_\alpha a \int_\beta \tau_\gamma^{-1}g\left(\int_\gamma \tau_\alpha^{-1}h \left(\int_\alpha f\right)\right) + \tau_\gamma\tau_\alpha a \int_\beta \tau_\alpha^{-1}g\left(\int_\alpha \tau_\gamma^{-1}f\left(\int_\gamma h\right)\right).
\end{multline*}
\end{example}

\begin{example} In this example we apply Algorithm~\ref{redalg} to a tree seen earlier.

\bigskip
Input:
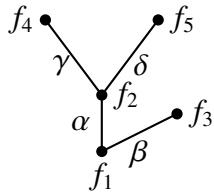
\[
\begin{tikzpicture}[thick]
\coordinate [label=below: $f_1$] (A) at (0,0);
\coordinate [label=right: $f_2$] (B) at (0,0.75);
\coordinate [label=right: $f_3$] (C) at (1,0.5);
\coordinate [label=left: $f_4$] (D) at (-0.75,1.75);
\coordinate [label=right: $f_5$] (E) at (0.75,1.75);
\draw (C) -- (A) -- (B);
\draw (E) -- (B) -- (D);
\foreach \n in {A,B,C,D,E} \node at (\n)[circle,fill,inner sep=1.5pt]{};
\coordinate [label=left: $\alpha$] (F) at (0,0.375);
\coordinate [label=below: $\beta$] (G) at (0.5,0.3);
\coordinate [label=left: $\gamma$] (H) at (-0.27,1.15);
\coordinate [label=right: $\delta$] (I) at (0.27,1.15);
\end{tikzpicture}
\]

Iteration 1:
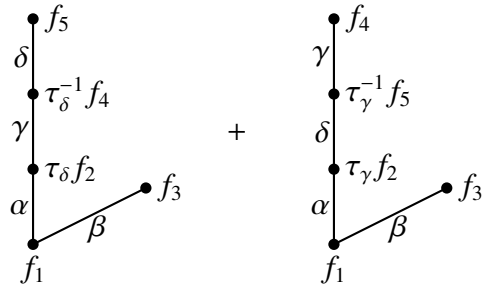
\[
\begin{tikzpicture}[thick]
\coordinate[label=below: $f_1$] (A1) at (0,0);
\coordinate[label=right: $\tau_\delta f_2$] (A2) at (0,1);
\coordinate[label=right: $\tau_\delta^{-1}f_4$] (A3) at (0,2);
\coordinate[label=right: $f_5$] (A4) at (0,3);
\coordinate[label=right: $f_3$] (A5) at (1.5,0.75);
\draw (A5) -- (A1) -- (A2) -- (A3) -- (A4);
\foreach \n in {A1,A2,A3,A4,A5} \node at (\n)[circle,fill,inner sep=1.5pt]{};
\coordinate[label=left: $\alpha$] (A6) at (0.1,0.5);
\coordinate[label=left: $\gamma$] (A7) at (0.1,1.5);
\coordinate[label=left: $\delta$] (A8) at (0.1,2.5);
\coordinate[label=right: $\beta$] (A9) at (0.6,0.2);
\coordinate[label=left: $+$] (A) at (3,1.5);
\coordinate[label=below: $f_1$] (B1) at (4,0);
\coordinate[label=right: $\tau_\gamma f_2$] (B2) at (4,1);
\coordinate[label=right: $\tau_\gamma^{-1}f_5$] (B3) at (4,2);
\coordinate[label=right: $f_4$] (B4) at (4,3);
\coordinate[label=right: $f_3$] (B5) at (5.5,0.75);
\draw (B5) -- (B1) -- (B2) -- (B3) -- (B4);
\foreach \n in {B1,B2,B3,B4,B5} \node at (\n)[circle,fill,inner sep=1.5pt]{};
\coordinate[label=left: $\alpha$] (B6) at (4.1,0.5);
\coordinate[label=left: $\delta$] (B7) at (4.1,1.5);
\coordinate[label=left: $\gamma$] (B8) at (4.1,2.5);
\coordinate[label=right: $\beta$] (B9) at (4.6,0.2);
\end{tikzpicture}
\]

Output:
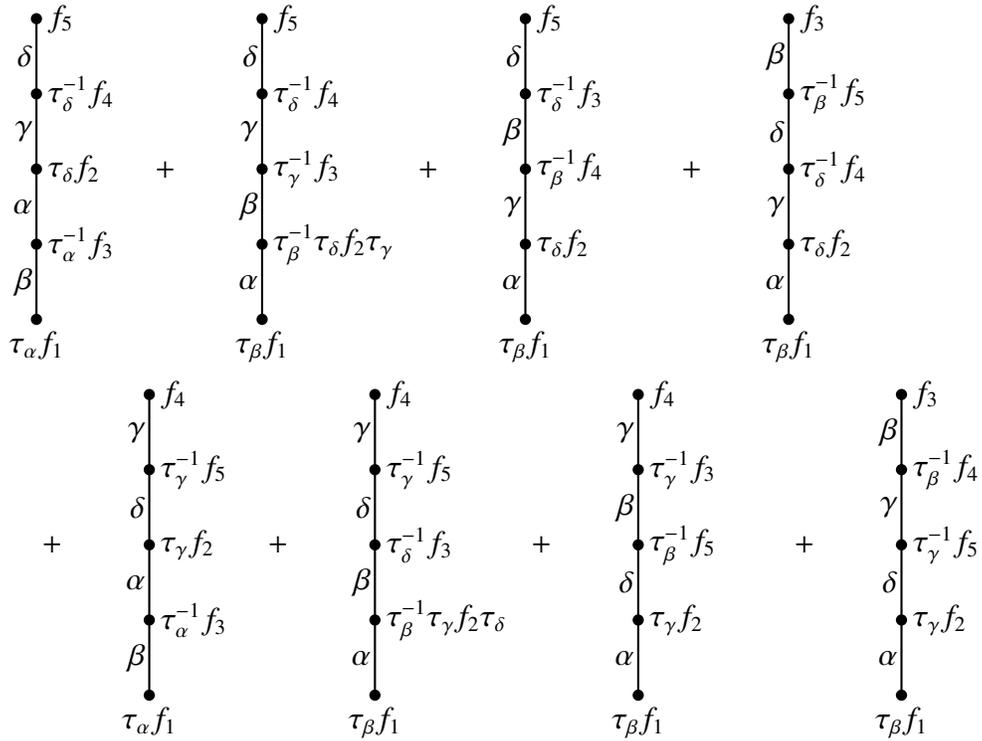
\[
\begin{tikzpicture}[thick]
\coordinate[label=below: $\tau_\alpha f_1$] (A1) at (0,0);
\coordinate[label=right: $\tau_\alpha^{-1}f_3$] (A2) at (0,1);
\coordinate[label=right: $\tau_\delta f_2$] (A3) at (0,2);
\coordinate[label=right: $\tau_\delta^{-1}f_4$](A4) at (0,3);
\coordinate[label=right: $f_5$] (A5) at (0,4);
\draw (A1) -- (A5);
\foreach \n in {A1,A2,A3,A4,A5} \node at (\n)[circle,fill,inner sep=1.5pt]{};
\coordinate[label=left: $\beta$] (A6) at (0.1,0.5);
\coordinate[label=left: $\alpha$] (A7) at (0.1,1.5);
\coordinate[label=left: $\gamma$] (A8) at (0.1,2.5);
\coordinate[label=left: $\delta$] (A9) at (0.1,3.5);
\coordinate[label=left: $+$] (A) at (2,2);
\coordinate[label=below: $\tau_\beta f_1$] (B1) at (3,0);
\coordinate[label=right: $\tau_\beta^{-1}\tau_\delta f_2 \tau_\gamma$] (B2) at (3,1);
\coordinate[label=right: $\tau_\gamma^{-1}f_3$] (B3) at (3,2);
\coordinate[label=right: $\tau_\delta^{-1}f_4$](B4) at (3,3);
\coordinate[label=right: $f_5$] (B5) at (3,4);
\draw (B1) -- (B5);
\foreach \n in {B1,B2,B3,B4,B5} \node at (\n)[circle,fill,inner sep=1.5pt]{};
\coordinate[label=left: $\alpha$] (B6) at (3.1,0.5);
\coordinate[label=left: $\beta$] (B7) at (3.1,1.5);
\coordinate[label=left: $\gamma$] (B8) at (3.1,2.5);
\coordinate[label=left: $\delta$] (B9) at (3.1,3.5);
\coordinate[label=left: $+$] (B) at (5.5,2);
\coordinate[label=below: $\tau_\beta f_1$] (C1) at (6.5,0);
\coordinate[label=right: $\tau_\delta f_2$] (C2) at (6.5,1);
\coordinate[label=right: $\tau_\beta^{-1}f_4$] (C3) at (6.5,2);
\coordinate[label=right: $\tau_\delta^{-1}f_3$](C4) at (6.5,3);
\coordinate[label=right: $f_5$] (C5) at (6.5,4);
\draw (C1) -- (C5);
\foreach \n in {C1,C2,C3,C4,C5} \node at (\n)[circle,fill,inner sep=1.5pt]{};
\coordinate[label=left: $\alpha$] (C6) at (6.6,0.5);
\coordinate[label=left: $\gamma$] (C7) at (6.6,1.5);
\coordinate[label=left: $\beta$] (C8) at (6.6,2.5);
\coordinate[label=left: $\delta$] (C9) at (6.6,3.5);
\coordinate[label=left: $+$] (C) at (9,2);
\coordinate[label=below: $\tau_\beta f_1$] (D1) at (10,0);
\coordinate[label=right: $\tau_\delta f_2$] (D2) at (10,1);
\coordinate[label=right: $\tau_\delta^{-1}f_4$] (D3) at (10,2);
\coordinate[label=right: $\tau_\beta^{-1}f_5$](D4) at (10,3);
\coordinate[label=right: $f_3$] (D5) at (10,4);
\draw (D1) -- (D5);
\foreach \n in {D1,D2,D3,D4,D5} \node at (\n)[circle,fill,inner sep=1.5pt]{};
\coordinate[label=left: $\alpha$] (D6) at (10.1,0.5);
\coordinate[label=left: $\gamma$] (D7) at (10.1,1.5);
\coordinate[label=left: $\delta$] (D8) at (10.1,2.5);
\coordinate[label=left: $\beta$] (D9) at (10.1,3.5);
\coordinate[label=left: $+$] (D) at (0.5,-3);
\coordinate[label=below: $\tau_\alpha f_1$] (E1) at (1.5,-5);
\coordinate[label=right: $\tau_\alpha^{-1}f_3$] (E2) at (1.5,-4);
\coordinate[label=right: $\tau_\gamma f_2$] (E3) at (1.5,-3);
\coordinate[label=right: $\tau_\gamma^{-1}f_5$](E4) at (1.5,-2);
\coordinate[label=right: $f_4$] (E5) at (1.5,-1);
\draw (E1) -- (E5);
\foreach \n in {E1,E2,E3,E4,E5} \node at (\n)[circle,fill,inner sep=1.5pt]{};
\coordinate[label=left: $\beta$] (E6) at (1.6,-4.5);
\coordinate[label=left: $\alpha$] (E7) at (1.6,-3.5);
\coordinate[label=left: $\delta$] (E8) at (1.6,-2.5);
\coordinate[label=left: $\gamma$] (E9) at (1.6,-1.5);
\coordinate[label=left: $+$] (E) at (3.5,-3);
\coordinate[label=below: $\tau_\beta f_1$] (F1) at (4.5,-5);
\coordinate[label=right: $\tau_\beta^{-1}\tau_\gamma f_2 \tau_\delta$] (F2) at (4.5,-4);
\coordinate[label=right: $\tau_\delta^{-1}f_3$] (F3) at (4.5,-3);
\coordinate[label=right: $\tau_\gamma^{-1}f_5$](F4) at (4.5,-2);
\coordinate[label=right: $f_4$] (F5) at (4.5,-1);
\draw (F1) -- (F5);
\foreach \n in {F1,F2,F3,F4,F5} \node at (\n)[circle,fill,inner sep=1.5pt]{};
\coordinate[label=left: $\alpha$] (F6) at (4.6,-4.5);
\coordinate[label=left: $\beta$] (F7) at (4.6,-3.5);
\coordinate[label=left: $\delta$] (F8) at (4.6,-2.5);
\coordinate[label=left: $\gamma$] (F9) at (4.6,-1.5);
\coordinate[label=left: $+$] (F) at (7,-3);
\coordinate[label=below: $\tau_\beta f_1$] (G1) at (8,-5);
\coordinate[label=right: $\tau_\gamma f_2$] (G2) at (8,-4);
\coordinate[label=right: $\tau_\beta^{-1}f_5$] (G3) at (8,-3);
\coordinate[label=right: $\tau_\gamma^{-1}f_3$](G4) at (8,-2);
\coordinate[label=right: $f_4$] (G5) at (8,-1);
\draw (G1) -- (G5);
\foreach \n in {G1,G2,G3,G4,G5} \node at (\n)[circle,fill,inner sep=1.5pt]{};
\coordinate[label=left: $\alpha$] (G6) at (8.1,-4.5);
\coordinate[label=left: $\delta$] (G7) at (8.1,-3.5);
\coordinate[label=left: $\beta$] (G8) at (8.1,-2.5);
\coordinate[label=left: $\gamma$] (G9) at (8.1,-1.5);
\coordinate[label=left: $+$] (G) at (10.5,-3);
\coordinate[label=below: $\tau_\beta f_1$] (H1) at (11.5,-5);
\coordinate[label=right: $\tau_\gamma f_2$] (H2) at (11.5,-4);
\coordinate[label=right: $\tau_\gamma^{-1}f_5$] (H3) at (11.5,-3);
\coordinate[label=right: $\tau_\beta^{-1}f_4$](H4) at (11.5,-2);
\coordinate[label=right: $f_3$] (H5) at (11.5,-1);
\draw (H1) -- (H5);
\foreach \n in {H1,H2,H3,H4,H5} \node at (\n)[circle,fill,inner sep=1.5pt]{};
\coordinate[label=left: $\alpha$] (H6) at (11.6,-4.5);
\coordinate[label=left: $\delta$] (H7) at (11.6,-3.5);
\coordinate[label=left: $\gamma$] (H8) at (11.6,-2.5);
\coordinate[label=left: $\beta$] (H9) at (11.6,-1.5);
\end{tikzpicture}
\]
As in the previous example, in this example following the first iteration of Algorithm~\ref{redalg} we applied the results of Example~\ref{iterated3} to complete the reduction.
\end{example}

\newpage
We conclude this section with a generalization of Examples~\ref{iterated2} and~\ref{iterated3}.
\begin{theorem}\label{iteratedlemma}
Let $a,f,g_1,\dots,g_m \in C(\R)$ and $\alpha,\beta_1,\dots,\beta_m \in \Omega$.  Then in $\mathfrak{E}(\Omega)$:
\begin{equation}\label{eq:iteratedlemma}
\begin{aligned}
\begin{tikzpicture}[thick]
\coordinate[label=below: $a$] (A1) at (0,0);
\coordinate[label=left: $f$] (A2) at (-1,1);
\coordinate[label=right: $g_1$] (A3) at (1,1);
\coordinate[label=right: $g_2$] (A4) at (1,2);
\coordinate[label=right: $g_{m-1}$] (A5) at (1,3);
\coordinate[label=right: $g_m$] (A6) at (1,4);
\draw (A2) -- (A1) -- (A3) -- (A4);
\draw[loosely dotted] (A4) -- (A5);
\draw (A5) -- (A6);
\foreach \n in {A1,A2,A3,A4,A5,A6} \node at (\n)[circle,fill,inner sep=1.5pt]{};
\coordinate[label=left: $\alpha$] (A7) at (-0.4,0.4);
\coordinate[label=left: $\beta_1$] (A8) at (0.7,0.7);
\coordinate[label=left: $\beta_2$] (A9) at (1.15,1.5);
\coordinate[label=left: $\beta_m$] (A10) at (1.15,3.5);
\coordinate[label=left: $\text{=}$] (A) at (3,2);
\coordinate[label=below: $\tau_{\beta_1}a$] (B1) at (4.5,-0.5);
\coordinate[label=right: $\tau_{\beta_1}^{-1}f$] (B2) at (4.5,0.5);
\coordinate[label=right: $g_1$] (B3) at (4.5,1.5);
\coordinate[label=right: $g_2$] (B4) at (4.5,2.5);
\coordinate[label=right: $g_{m-1}$] (B5) at (4.5,3.5);
\coordinate[label=right: $g_m$] (B6) at (4.5,4.5);
\draw (B1) -- (B4);
\draw[loosely dotted] (B4) -- (B5);
\draw (B5) -- (B6);
\foreach \n in {B1,B2,B3,B4,B5,B6} \node at (\n)[circle,fill,inner sep=1.5pt]{};
\coordinate[label=left: $\alpha$] (B7) at (4.6,0);
\coordinate[label=left: $\beta_1$] (B8) at (4.65,1);
\coordinate[label=left: $\beta_2$] (B9) at (4.65,2);
\coordinate[label=left: $\beta_m$] (B10) at (4.65,4);
\coordinate[label=left: $+$] (B) at (6,2);
\coordinate[label=left: $\displaystyle \sum_{i=1}^{m-1}$] (B') at (7,2);
\coordinate[label=below: $\tau_\alpha a$] (C1) at (7.5,-2);
\coordinate[label=right: $g_1$] (C2) at (7.5,-1);
\coordinate[label=right: $g_2$] (C3) at (7.5,0);
\coordinate[label=right: $g_{i-1}$] (C4) at (7.5,1);
\coordinate[label=right: $\tau_\alpha^{-1}g_i\tau_{\beta_{i+1}}$] (C5) at (7.5,2);
\coordinate[label=right: $\tau_{\beta_{i+1}}^{-1}f$] (C6) at (7.5,3);
\coordinate[label=right: $g_{i+1}$] (C7) at (7.5,4);
\coordinate[label=right: $g_{m-1}$] (C8) at (7.5,5);
\coordinate[label=right: $g_m$] (C9) at (7.5,6);
\draw (C1) -- (C3);
\draw[loosely dotted] (C3) -- (C4);
\draw (C4) -- (C7);
\draw[loosely dotted] (C7) -- (C8);
\draw (C8) -- (C9);
\foreach \n in {C1,C2,C3,C4,C5,C6,C7,C8,C9} \node at (\n)[circle,fill,inner sep=1.5pt]{};
\coordinate[label=left: $\beta_1$] (C10) at (7.65,-1.5);
\coordinate[label=left: $\beta_2$] (C11) at (7.65,-0.5);
\coordinate[label=left: $\beta_i$] (C12) at (7.65,1.5);
\coordinate[label=left: $\alpha$] (C13) at (7.6,2.5);
\coordinate[label=left: $\beta_{i+1}$] (C14) at (7.65,3.5);
\coordinate[label=left: $\beta_m$] (C15) at (7.65,5.5);
\coordinate[label=left: $+$] (C) at (10,2);
\coordinate[label=below: $\tau_\alpha a$] (D1) at (11,-0.5);
\coordinate[label=right: $g_1$] (D2) at (11,0.5);
\coordinate[label=right: $g_2$] (D3) at (11,1.5);
\coordinate[label=right: $g_{m-1}$] (D4) at (11,2.5);
\coordinate[label=right: $\tau_\alpha^{-1}g_m$] (D5) at (11,3.5);
\coordinate[label=right: $f$] (D6) at (11,4.5);
\draw (D1) -- (D3);
\draw[loosely dotted] (D3) -- (D4);
\draw (D4) -- (D6);
\foreach \n in {D1,D2,D3,D4,D5,D6} \node at (\n)[circle,fill,inner sep=1.5pt]{};
\coordinate[label=left: $\beta_1$] (D7) at (11.15,0);
\coordinate[label=left: $\beta_2$] (D8) at (11.15,1);
\coordinate[label=left: $\beta_m$] (D9) at (11.15,3);
\coordinate[label=left: $\alpha$] (D10) at (11.1,4);
\end{tikzpicture}
\end{aligned}
\end{equation}
\end{theorem}

In integral notation, the left-hand side of Eq.~\eqref{eq:iteratedlemma} is the produce of a single integral with an iterated integral containing $m$ iterates, so the equation translates to
\begin{multline*}
a\left(\int_\alpha f\right)\left(\int_{\beta_1}g_1\left(\int_{\beta_2}g_2\cdots\left(\int_{\beta_m}g_m\right)\cdots\right)\right) \\
= \tau_{\beta_1}a\int_\alpha \tau_{\beta_1}^{-1}f\left(\int_{\beta_1}g_1\left(\int_{\beta_2}g_2\cdots\left(\int_{\beta_m}g_m\right)\cdots\right)\right) \\
+ \sum_{i=1}^{m-1} \tau_\alpha a \int_{\beta_1}g_1\left(\int_{\beta_2}g_2\cdots\left(\int_{\beta_{i-1}}g_{i-1}\left(\int_{\beta_i}\tau_\alpha^{-1}g_i\tau_{\beta_{i+1}}\left(\int_\alpha \tau_{\beta_{i+1}}^{-1}f\left(\int_{\beta_{i+1}} g_{i+1}\cdots \left(\int_{\beta_m}g_m\right)\cdots\right)\right)\right)\right)\cdots\right)\\
+ \tau_\alpha a\int_{\beta_1}g_1\left(\int_{\beta_2}g_2\cdots\left(\int_{\beta_{m-1}}g_{m-1}\left(\int_{\beta_m}\tau_\alpha^{-1}g_m\left(\int_\alpha f\right)\right)\right)\cdots\right).
\end{multline*}

\begin{proof} We will prove this by induction on $m$.  In the base case, we have the tree

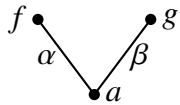
\[
\begin{tikzpicture}[thick]
\coordinate [label=right: $a$] (A) at (0,0);
\coordinate [label=left: $f$] (B) at (-0.75,1);
\coordinate [label=right: $g$] (C) at (0.75,1);
\coordinate [label=left: $\alpha$] (D) at (-0.35,0.5);
\coordinate [label=right: $\beta$] (E) at (0.35,0.5);
\draw (B) -- (A) -- (C);
\foreach \n in {A,B,C} \node at (\n)[circle,fill,inner sep=1.5pt]{};
\end{tikzpicture}
\]
Here we can apply Theorem~\ref{theorem:RBtrees} to obtain

\[
\begin{tikzpicture}[thick]
\coordinate [label=right: $a$] (A) at (0,0);
\coordinate [label=left: $f$] (B) at (-0.75,1);
\coordinate [label=right: $g$] (C) at (0.75,1);
\coordinate [label=left: $\alpha$] (D) at (-0.35,0.5);
\coordinate [label=right: $\beta$] (E) at (0.35,0.5);
\draw (B) -- (A) -- (C);
\coordinate [label=left: $\text{=}$] (F) at (2.25,0.5);
\coordinate [label=right: $\tau_\beta a$] (A1) at (3,-0.5);
\coordinate [label=right: $\tau_\beta^{-1} f$] (A2) at (3,0.5);
\coordinate [label=right: $g$] (A3) at (3,1.5);
\coordinate [label=left: $\alpha$] (A4) at (3,0);
\coordinate [label=left: $\beta$] (A5) at (3,1);
\draw (A1) -- (A2) -- (A3);
\coordinate [label=left: $+$] (G) at (5,0.5);
\coordinate [label=right: $\tau_\alpha a$] (B1) at (5.75,-0.5);
\coordinate [label=right: $\tau_\alpha^{-1} g$] (B2) at (5.75,0.5);
\coordinate [label=right: $f$] (B3) at (5.75,1.5);
\coordinate [label=left: $\beta$] (B4) at (5.75,0);
\coordinate [label=left: $\alpha$] (B5) at (5.75,1);
\draw (B1) -- (B2) -- (B3);
\foreach \n in {A,B,C,A1,A2,A3,B1,B2,B3} \node at (\n)[circle,fill,inner sep=1.5pt]{};
\end{tikzpicture}
\]
This agrees with Eq.~\eqref{eq:iteratedlemma} when $m=1$ (note when $m=1$, the middle term of the sum on the right-hand-side of Eq.~\eqref{eq:iteratedlemma} is empty).

For the inductive step, let $m > 1$ and consider the following tree:
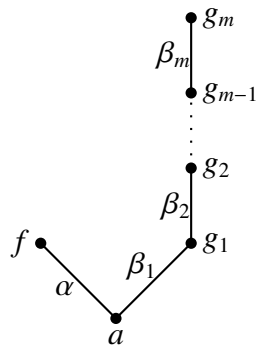
\[
\begin{tikzpicture}[thick]
\coordinate[label=below: $a$] (A1) at (0,0);
\coordinate[label=left: $f$] (A2) at (-1,1);
\coordinate[label=right: $g_1$] (A3) at (1,1);
\coordinate[label=right: $g_2$] (A4) at (1,2);
\coordinate[label=right: $g_{m-1}$] (A5) at (1,3);
\coordinate[label=right: $g_m$] (A6) at (1,4);
\draw (A2) -- (A1) -- (A3) -- (A4);
\draw[loosely dotted] (A4) -- (A5);
\draw (A5) -- (A6);
\foreach \n in {A1,A2,A3,A4,A5,A6} \node at (\n)[circle,fill,inner sep=1.5pt]{};
\coordinate[label=left: $\alpha$] (A7) at (-0.4,0.4);
\coordinate[label=left: $\beta_1$] (A8) at (0.7,0.7);
\coordinate[label=left: $\beta_2$] (A9) at (1.15,1.5);
\coordinate[label=left: $\beta_m$] (A10) at (1.15,3.5);
\end{tikzpicture}
\]
After applying Theorem~\ref{theorem:RBtrees} once, the above tree is equivalent to the following:
\begin{equation}\label{lemma:inductivestep}
\begin{aligned}
\begin{tikzpicture}[thick]
\coordinate[label=below: $\tau_{\beta_1}a$] (A1) at (0,0);
\coordinate[label=right: $\tau_{\beta_1}^{-1}f$] (A2) at (0,1);
\coordinate[label=right: $g_1$] (A3) at (0,2);
\coordinate[label=right: $g_2$] (A4) at (0,3);
\coordinate[label=right: $g_{m-1}$] (A5) at (0,4);
\coordinate[label=right: $g_m$] (A6) at (0,5);
\draw (A1) -- (A4);
\draw[loosely dotted] (A4) -- (A5);
\draw (A5) -- (A6);
\foreach \n in {A1,A2,A3,A4,A5,A6} \node at (\n)[circle,fill,inner sep=1.5pt]{};
\coordinate[label=left: $\alpha$] (A7) at (0.1,0.5);
\coordinate[label=left: $\beta_1$] (A8) at (0.15,1.5);
\coordinate[label=left: $\beta_2$] (A9) at (0.15,2.5);
\coordinate[label=left: $\beta_m$] (A10) at (0.15,4.5);
\coordinate[label=left: $+$] (A) at (2,2.5);
\coordinate[label=below: $\tau_\alpha a$] (B1) at (4,0);
\coordinate[label=right: $\tau_\alpha^{-1}g_1$] (B2) at (4,1);
\coordinate[label=left: $f$] (B3) at (3,2);
\coordinate[label=right: $g_2$] (B4) at (5,2);
\coordinate[label=right: $g_3$] (B5) at (5,3);
\coordinate[label=right: $g_{m-1}$] (B6) at (5,4);
\coordinate[label=right: $g_m$] (B7) at (5,5);
\draw (B1) -- (B2) -- (B4) -- (B5);
\draw (B2) -- (B3);
\draw[loosely dotted] (B5) -- (B6);
\draw (B6) -- (B7);
\foreach \n in {B1,B2,B3,B4,B5,B6,B7} \node at (\n)[circle,fill,inner sep=1.5pt]{};
\coordinate[label=left: $\beta_1$] (B8) at (4.15,0.5);
\coordinate[label=left: $\alpha$] (B9) at (3.6,1.4);
\coordinate[label=left: $\beta_2$] (B10) at (4.7,1.7);
\coordinate[label=left: $\beta_3$] (B11) at (5.15,2.5);
\coordinate[label=left: $\beta_m$] (B12) at (5.15,4.5);
\end{tikzpicture}
\end{aligned}
\end{equation}

By the inductive hypothesis,
\[
\begin{tikzpicture}[thick]
\coordinate[label=below: $\tau_\alpha^{-1}g_1$] (A1) at (0,0);
\coordinate[label=left: $f$] (A2) at (-1,1);
\coordinate[label=right: $g_2$] (A3) at (1,1);
\coordinate[label=right: $g_3$] (A4) at (1,2);
\coordinate[label=right: $g_{m-1}$] (A5) at (1,3);
\coordinate[label=right: $g_m$] (A6) at (1,4);
\draw (A2) -- (A1) -- (A3) -- (A4);
\draw[loosely dotted] (A4) -- (A5);
\draw (A5) -- (A6);
\foreach \n in {A1,A2,A3,A4,A5,A6} \node at (\n)[circle,fill,inner sep=1.5pt]{};
\coordinate[label=left: $\alpha$] (A7) at (-0.4,0.4);
\coordinate[label=left: $\beta_2$] (A8) at (0.7,0.7);
\coordinate[label=left: $\beta_3$] (A9) at (1.15,1.5);
\coordinate[label=left: $\beta_m$] (A10) at (1.15,3.5);
\coordinate[label=left: $\text{=}$] (A) at (3,2);
\coordinate[label=below: $\tau_\alpha^{-1}g_1\tau_{\beta_2}$] (B1) at (4.5,-0.5);
\coordinate[label=right: $\tau_{\beta_2}^{-1}f$] (B2) at (4.5,0.5);
\coordinate[label=right: $g_2$] (B3) at (4.5,1.5);
\coordinate[label=right: $g_3$] (B4) at (4.5,2.5);
\coordinate[label=right: $g_{m-1}$] (B5) at (4.5,3.5);
\coordinate[label=right: $g_m$] (B6) at (4.5,4.5);
\draw (B1) -- (B4);
\draw[loosely dotted] (B4) -- (B5);
\draw (B5) -- (B6);
\foreach \n in {B1,B2,B3,B4,B5,B6} \node at (\n)[circle,fill,inner sep=1.5pt]{};
\coordinate[label=left: $\alpha$] (B7) at (4.6,0);
\coordinate[label=left: $\beta_2$] (B8) at (4.65,1);
\coordinate[label=left: $\beta_3$] (B9) at (4.65,2);
\coordinate[label=left: $\beta_m$] (B10) at (4.65,4);
\coordinate[label=left: $+$] (B) at (6,2);
\coordinate[label=left: $\displaystyle \sum_{i=2}^{m-1}$] (B') at (7,2);
\coordinate[label=below: $g_1$] (C1) at (7.5,-2);
\coordinate[label=right: $g_2$] (C2) at (7.5,-1);
\coordinate[label=right: $g_3$] (C3) at (7.5,0);
\coordinate[label=right: $g_{i-1}$] (C4) at (7.5,1);
\coordinate[label=right: $\tau_\alpha^{-1}g_i\tau_{\beta_{i+1}}$] (C5) at (7.5,2);
\coordinate[label=right: $\tau_{\beta_{i+1}}^{-1}f$] (C6) at (7.5,3);
\coordinate[label=right: $g_{i+1}$] (C7) at (7.5,4);
\coordinate[label=right: $g_{m-1}$] (C8) at (7.5,5);
\coordinate[label=right: $g_m$] (C9) at (7.5,6);
\draw (C1) -- (C3);
\draw[loosely dotted] (C3) -- (C4);
\draw (C4) -- (C7);
\draw[loosely dotted] (C7) -- (C8);
\draw (C8) -- (C9);
\foreach \n in {C1,C2,C3,C4,C5,C6,C7,C8,C9} \node at (\n)[circle,fill,inner sep=1.5pt]{};
\coordinate[label=left: $\beta_2$] (C10) at (7.65,-1.5);
\coordinate[label=left: $\beta_3$] (C11) at (7.65,-0.5);
\coordinate[label=left: $\beta_i$] (C12) at (7.65,1.5);
\coordinate[label=left: $\alpha$] (C13) at (7.6,2.5);
\coordinate[label=left: $\beta_{i+1}$] (C14) at (7.65,3.5);
\coordinate[label=left: $\beta_m$] (C15) at (7.65,5.5);
\coordinate[label=left: $+$] (C) at (10,2);
\coordinate[label=below: $g_1$] (D1) at (11,-0.5);
\coordinate[label=right: $g_2$] (D2) at (11,0.5);
\coordinate[label=right: $g_3$] (D3) at (11,1.5);
\coordinate[label=right: $g_{m-1}$] (D4) at (11,2.5);
\coordinate[label=right: $\tau_\alpha^{-1}g_m$] (D5) at (11,3.5);
\coordinate[label=right: $f$] (D6) at (11,4.5);
\draw (D1) -- (D3);
\draw[loosely dotted] (D3) -- (D4);
\draw (D4) -- (D6);
\foreach \n in {D1,D2,D3,D4,D5,D6} \node at (\n)[circle,fill,inner sep=1.5pt]{};
\coordinate[label=left: $\beta_2$] (D7) at (11.15,0);
\coordinate[label=left: $\beta_3$] (D8) at (11.15,1);
\coordinate[label=left: $\beta_m$] (D9) at (11.15,3);
\coordinate[label=left: $\alpha$] (D10) at (11.1,4);
\end{tikzpicture}
\]
Reattaching the first edge from the second term in Eq.~\eqref{lemma:inductivestep}, and adding in the first term, Eq.~\eqref{lemma:inductivestep} becomes
\begin{equation}\label{lemma:inductivestep2}
\begin{aligned}
\begin{tikzpicture}[thick]
\coordinate[label=below: $\tau_{\beta_1}a$] (A1) at (2,-1);
\coordinate[label=right: $\tau_{\beta_1}^{-1}f$] (A2) at (2,0);
\coordinate[label=right: $g_1$] (A3) at (2,1);
\coordinate[label=right: $g_2$] (A4) at (2,2);
\coordinate[label=right: $g_{m-1}$] (A5) at (2,3);
\coordinate[label=right: $g_m$] (A6) at (2,4);
\draw (A1) -- (A4);
\draw[loosely dotted] (A4) -- (A5);
\draw (A5) -- (A6);
\foreach \n in {A1,A2,A3,A4,A5,A6} \node at (\n)[circle,fill,inner sep=1.5pt]{};
\coordinate[label=left: $\alpha$] (A7) at (2.1,-0.5);
\coordinate[label=left: $\beta_1$] (A8) at (2.15,0.5);
\coordinate[label=left: $\beta_2$] (A9) at (2.15,1.5);
\coordinate[label=left: $\beta_m$] (A10) at (2.15,3.5);
\coordinate[label=left: $+$] (A) at (3.5,1.5);
\coordinate[label=below: $\tau_\alpha a$] (B0) at (4.5,-1.5);
\coordinate[label=right: $\tau_\alpha^{-1}g_1\tau_{\beta_2}$] (B1) at (4.5,-0.5);
\coordinate[label=right: $\tau_{\beta_2}^{-1}f$] (B2) at (4.5,0.5);
\coordinate[label=right: $g_2$] (B3) at (4.5,1.5);
\coordinate[label=right: $g_3$] (B4) at (4.5,2.5);
\coordinate[label=right: $g_{m-1}$] (B5) at (4.5,3.5);
\coordinate[label=right: $g_m$] (B6) at (4.5,4.5);
\draw (B0) -- (B4);
\draw[loosely dotted] (B4) -- (B5);
\draw (B5) -- (B6);
\foreach \n in {B0,B1,B2,B3,B4,B5,B6} \node at (\n)[circle,fill,inner sep=1.5pt]{};
\coordinate[label=left: $\alpha$] (B7) at (4.6,0);
\coordinate[label=left: $\beta_2$] (B8) at (4.65,1);
\coordinate[label=left: $\beta_3$] (B9) at (4.65,2);
\coordinate[label=left: $\beta_m$] (B10) at (4.65,4);
\coordinate[label=left: $\beta_1$] (B11) at (4.65,-1);
\coordinate[label=left: $+$] (B) at (6,1.5);
\coordinate[label=left: $\displaystyle \sum_{i=2}^{m-1}$] (B') at (7,1.5);
\coordinate[label=below: $\tau_\alpha a$] (C0) at (7.5,-3);
\coordinate[label=right: $g_1$] (C1) at (7.5,-2);
\coordinate[label=right: $g_2$] (C2) at (7.5,-1);
\coordinate[label=right: $g_3$] (C3) at (7.5,0);
\coordinate[label=right: $g_{i-1}$] (C4) at (7.5,1);
\coordinate[label=right: $\tau_\alpha^{-1}g_i\tau_{\beta_{i+1}}$] (C5) at (7.5,2);
\coordinate[label=right: $\tau_{\beta_{i+1}}^{-1}f$] (C6) at (7.5,3);
\coordinate[label=right: $g_{i+1}$] (C7) at (7.5,4);
\coordinate[label=right: $g_{m-1}$] (C8) at (7.5,5);
\coordinate[label=right: $g_m$] (C9) at (7.5,6);
\draw (C0) -- (C3);
\draw[loosely dotted] (C3) -- (C4);
\draw (C4) -- (C7);
\draw[loosely dotted] (C7) -- (C8);
\draw (C8) -- (C9);
\foreach \n in {C0,C1,C2,C3,C4,C5,C6,C7,C8,C9} \node at (\n)[circle,fill,inner sep=1.5pt]{};
\coordinate[label=left: $\beta_2$] (C10) at (7.65,-1.5);
\coordinate[label=left: $\beta_3$] (C11) at (7.65,-0.5);
\coordinate[label=left: $\beta_i$] (C12) at (7.65,1.5);
\coordinate[label=left: $\alpha$] (C13) at (7.6,2.5);
\coordinate[label=left: $\beta_{i+1}$] (C14) at (7.65,3.5);
\coordinate[label=left: $\beta_m$] (C15) at (7.65,5.5);
\coordinate[label=left: $\beta_1$] (C16) at (7.65,-2.5);
\coordinate[label=left: $+$] (C) at (10,1.5);
\coordinate[label=below: $\tau_\alpha a$] (D0) at (11,-1.5);
\coordinate[label=right: $g_1$] (D1) at (11,-0.5);
\coordinate[label=right: $g_2$] (D2) at (11,0.5);
\coordinate[label=right: $g_3$] (D3) at (11,1.5);
\coordinate[label=right: $g_{m-1}$] (D4) at (11,2.5);
\coordinate[label=right: $\tau_\alpha^{-1}g_m$] (D5) at (11,3.5);
\coordinate[label=right: $f$] (D6) at (11,4.5);
\draw (D0) -- (D3);
\draw[loosely dotted] (D3) -- (D4);
\draw (D4) -- (D6);
\foreach \n in {D0,D1,D2,D3,D4,D5,D6} \node at (\n)[circle,fill,inner sep=1.5pt]{};
\coordinate[label=left: $\beta_2$] (D7) at (11.15,0);
\coordinate[label=left: $\beta_3$] (D8) at (11.15,1);
\coordinate[label=left: $\beta_m$] (D9) at (11.15,3);
\coordinate[label=left: $\alpha$] (D10) at (11.1,4);
\coordinate[label=left: $\beta_1$] (D11) at (11.15,-1);
\end{tikzpicture}
\end{aligned}
\end{equation}
Notice that 
\[
\begin{tikzpicture}[thick]
\coordinate[label=below: $\tau_\alpha a$] (B0) at (4.5,-1.5);
\coordinate[label=right: $\tau_\alpha^{-1}g_1\tau_{\beta_2}$] (B1) at (4.5,-0.5);
\coordinate[label=right: $\tau_{\beta_2}^{-1}f$] (B2) at (4.5,0.5);
\coordinate[label=right: $g_2$] (B3) at (4.5,1.5);
\coordinate[label=right: $g_3$] (B4) at (4.5,2.5);
\coordinate[label=right: $g_{m-1}$] (B5) at (4.5,3.5);
\coordinate[label=right: $g_m$] (B6) at (4.5,4.5);
\draw (B0) -- (B4);
\draw[loosely dotted] (B4) -- (B5);
\draw (B5) -- (B6);
\foreach \n in {B0,B1,B2,B3,B4,B5,B6} \node at (\n)[circle,fill,inner sep=1.5pt]{};
\coordinate[label=left: $\alpha$] (B7) at (4.6,0);
\coordinate[label=left: $\beta_2$] (B8) at (4.65,1);
\coordinate[label=left: $\beta_3$] (B9) at (4.65,2);
\coordinate[label=left: $\beta_m$] (B10) at (4.65,4);
\coordinate[label=left: $\beta_1$] (B11) at (4.65,-1);
\coordinate[label=left: $+$] (B) at (6,1.5);
\coordinate[label=left: $\displaystyle \sum_{i=2}^{m-1}$] (B') at (7,1.5);
\coordinate[label=below: $\tau_\alpha a$] (C0) at (7.5,-3);
\coordinate[label=right: $g_1$] (C1) at (7.5,-2);
\coordinate[label=right: $g_2$] (C2) at (7.5,-1);
\coordinate[label=right: $g_3$] (C3) at (7.5,0);
\coordinate[label=right: $g_{i-1}$] (C4) at (7.5,1);
\coordinate[label=right: $\tau_\alpha^{-1}g_i\tau_{\beta_{i+1}}$] (C5) at (7.5,2);
\coordinate[label=right: $\tau_{\beta_{i+1}}^{-1}f$] (C6) at (7.5,3);
\coordinate[label=right: $g_{i+1}$] (C7) at (7.5,4);
\coordinate[label=right: $g_{m-1}$] (C8) at (7.5,5);
\coordinate[label=right: $g_m$] (C9) at (7.5,6);
\draw (C0) -- (C3);
\draw[loosely dotted] (C3) -- (C4);
\draw (C4) -- (C7);
\draw[loosely dotted] (C7) -- (C8);
\draw (C8) -- (C9);
\foreach \n in {C0,C1,C2,C3,C4,C5,C6,C7,C8,C9} \node at (\n)[circle,fill,inner sep=1.5pt]{};
\coordinate[label=left: $\beta_2$] (C10) at (7.65,-1.5);
\coordinate[label=left: $\beta_3$] (C11) at (7.65,-0.5);
\coordinate[label=left: $\beta_i$] (C12) at (7.65,1.5);
\coordinate[label=left: $\alpha$] (C13) at (7.6,2.5);
\coordinate[label=left: $\beta_{i+1}$] (C14) at (7.65,3.5);
\coordinate[label=left: $\beta_m$] (C15) at (7.65,5.5);
\coordinate[label=left: $\beta_1$] (C16) at (7.65,-2.5);
\coordinate[label=left: $\text{=}$] (C) at (10,1.5);
\coordinate[label=left: $\displaystyle \sum_{i=1}^{m-1}$] (C') at (11.5,1.5);
\coordinate[label=below: $\tau_\alpha a$] (D1) at (12,-2.5);
\coordinate[label=right: $g_1$] (D2) at (12,-1.5);
\coordinate[label=right: $g_2$] (D3) at (12,-0.5);
\coordinate[label=right: $g_{i-1}$] (D4) at (12,0.5);
\coordinate[label=right: $\tau_\alpha^{-1}g_i\tau_{\beta_{i+1}}$] (D5) at (12,1.5);
\coordinate[label=right: $\tau_{\beta_{i+1}}^{-1}f$] (D6) at (12,2.5);
\coordinate[label=right: $g_{i+1}$] (D7) at (12,3.5);
\coordinate[label=right: $g_{m-1}$] (D8) at (12,4.5);
\coordinate[label=right: $g_m$] (D9) at (12,5.5);
\draw (D1) -- (D3);
\draw[loosely dotted] (D3) -- (D4);
\draw (D4) -- (D7);
\draw[loosely dotted] (D7) -- (D8);
\draw (D8) -- (D9);
\foreach \n in {D1,D2,D3,D4,D5,D6,D7,D8,D9} \node at (\n)[circle,fill,inner sep=1.5pt]{};
\coordinate[label=left: $\beta_1$] (D10) at (12.15,-2);
\coordinate[label=left: $\beta_2$] (D11) at (12.15,-1);
\coordinate[label=left: $\beta_i$] (D12) at (12.15,1);
\coordinate[label=left: $\alpha$] (D13) at (12.1,2);
\coordinate[label=left: $\beta_{i+1}$] (D14) at (12.15,3);
\coordinate[label=left: $\beta_m$] (D15) at (12.15,5);
\end{tikzpicture}
\]
Thus Eq.~\eqref{lemma:inductivestep2} is equivalent to
\[
\begin{tikzpicture}[thick]
\coordinate[label=below: $\tau_{\beta_1}a$] (B1) at (4.5,-0.5);
\coordinate[label=right: $\tau_{\beta_1}^{-1}f$] (B2) at (4.5,0.5);
\coordinate[label=right: $g_1$] (B3) at (4.5,1.5);
\coordinate[label=right: $g_2$] (B4) at (4.5,2.5);
\coordinate[label=right: $g_{m-1}$] (B5) at (4.5,3.5);
\coordinate[label=right: $g_m$] (B6) at (4.5,4.5);
\draw (B1) -- (B4);
\draw[loosely dotted] (B4) -- (B5);
\draw (B5) -- (B6);
\foreach \n in {B1,B2,B3,B4,B5,B6} \node at (\n)[circle,fill,inner sep=1.5pt]{};
\coordinate[label=left: $\alpha$] (B7) at (4.6,0);
\coordinate[label=left: $\beta_1$] (B8) at (4.65,1);
\coordinate[label=left: $\beta_2$] (B9) at (4.65,2);
\coordinate[label=left: $\beta_m$] (B10) at (4.65,4);
\coordinate[label=left: $+$] (B) at (6,2);
\coordinate[label=left: $\displaystyle \sum_{i=1}^{m-1}$] (B') at (7,2);
\coordinate[label=below: $\tau_\alpha a$] (C1) at (7.5,-2);
\coordinate[label=right: $g_1$] (C2) at (7.5,-1);
\coordinate[label=right: $g_2$] (C3) at (7.5,0);
\coordinate[label=right: $g_{i-1}$] (C4) at (7.5,1);
\coordinate[label=right: $\tau_\alpha^{-1}g_i\tau_{\beta_{i+1}}$] (C5) at (7.5,2);
\coordinate[label=right: $\tau_{\beta_{i+1}}^{-1}f$] (C6) at (7.5,3);
\coordinate[label=right: $g_{i+1}$] (C7) at (7.5,4);
\coordinate[label=right: $g_{m-1}$] (C8) at (7.5,5);
\coordinate[label=right: $g_m$] (C9) at (7.5,6);
\draw (C1) -- (C3);
\draw[loosely dotted] (C3) -- (C4);
\draw (C4) -- (C7);
\draw[loosely dotted] (C7) -- (C8);
\draw (C8) -- (C9);
\foreach \n in {C1,C2,C3,C4,C5,C6,C7,C8,C9} \node at (\n)[circle,fill,inner sep=1.5pt]{};
\coordinate[label=left: $\beta_1$] (C10) at (7.65,-1.5);
\coordinate[label=left: $\beta_2$] (C11) at (7.65,-0.5);
\coordinate[label=left: $\beta_i$] (C12) at (7.65,1.5);
\coordinate[label=left: $\alpha$] (C13) at (7.6,2.5);
\coordinate[label=left: $\beta_{i+1}$] (C14) at (7.65,3.5);
\coordinate[label=left: $\beta_m$] (C15) at (7.65,5.5);
\coordinate[label=left: $+$] (C) at (10,2);
\coordinate[label=below: $\tau_\alpha a$] (D1) at (11,-0.5);
\coordinate[label=right: $g_1$] (D2) at (11,0.5);
\coordinate[label=right: $g_2$] (D3) at (11,1.5);
\coordinate[label=right: $g_{m-1}$] (D4) at (11,2.5);
\coordinate[label=right: $\tau_\alpha^{-1}g_m$] (D5) at (11,3.5);
\coordinate[label=right: $f$] (D6) at (11,4.5);
\draw (D1) -- (D3);
\draw[loosely dotted] (D3) -- (D4);
\draw (D4) -- (D6);
\foreach \n in {D1,D2,D3,D4,D5,D6} \node at (\n)[circle,fill,inner sep=1.5pt]{};
\coordinate[label=left: $\beta_1$] (D7) at (11.15,0);
\coordinate[label=left: $\beta_2$] (D8) at (11.15,1);
\coordinate[label=left: $\beta_m$] (D9) at (11.15,3);
\coordinate[label=left: $\alpha$] (D10) at (11.1,4);
\end{tikzpicture}
\]
This is simply the right-hand side of Eq.~\eqref{eq:iteratedlemma}, completing the proof.
\end{proof}


\section{Conclusion}
Our goal in this project was to construct an algorithm which reduces integral equations with products to their equivalent operator linear form.  Previously such a reduction was done on a case-by-case basis and had been proved possible in general using non-constructive methods.  The algorithm we construct gives us a systematic approach for this reduction that can be applied to any Volterra integral equation with separable kernels.  This work has several possibilities for continuation.  We would like to study the computational complexity of the algorithm, e.g. to determine the number of steps of the algorithm based on the data of the input forest.  In addition, we plan to adapt the algorithm to more general classes of integral equations.  For example, sum-separable kernels, i.e. kernels that can be written as a sum of separable terms such as $\cos(x-t)=\cos(x)\cos(t)+\sin(x)\sin(t)$, appear often in applications.  The integral operator with sum-separable kernels does not satisfy the twisted Rota-Baxter identity that it does when it contains separable kernels, and thus a new approach must be taken to handle such a case.

\section*{Acknowledgements}
This research was partially sponsored by the Manhattan College Jasper Scholar Summer Research Scholars program.

\end{document}